%% file: article_definitive_version.tex
\documentclass{amsart}

\input{packages}

\title[Deformations of twisted harmonic maps]{Deformations of twisted harmonic maps and variation of the energy}
\author{Marco ~Spinaci}
\address{Institut Fourier, Grenoble.}
\date{\today}

\begin{document}

\begin{abstract}
  We study the deformations of twisted harmonic maps $f$ with respect to the representation $\rho$. After constructing a continuous ``universal'' twisted harmonic map, we give a construction of every first order deformation of $f$ in terms of Hodge theory; we apply this result to the moduli space of reductive representations of a K\"ahler group, to show that the critical points of the energy functional $E$ coincide with the monodromy representations of polarized complex variations of Hodge structure. We then proceed to second order deformations, where obstructions arise; we investigate the existence of such deformations, and give a method for constructing them, as well. Applying this to the energy functional as above, we prove (for every finitely presented group) that the energy functional is a potential for the K\"ahler form of the ``Betti'' moduli space; assuming furthermore that the group is K\"ahler, we study the eigenvalues of the Hessian of $E$ at critical points.
\end{abstract}

\maketitle

\section*{Introduction}

Harmonic maps have a long history which dates back at least to 1964, when Eells and Sampson \cite{EeSa64} proved the existence of a harmonic representative in every homotopy class of maps between compact manifolds of appropriate curvature. Precise results about uniqueness and the variation of the energy have followed, in \cite{Ha67} and \cite{Ma73}, respectively. It became evident that harmonic maps enjoy especially good properties if one supposes in addition the starting manifold $X$ to be K\"ahler; this is resumed in the Siu-Sampson Bochner's formula, \cite{Si80,Sa86}, which implies that the harmonic map is in fact pluriharmonic.

While these concepts were perfectioned, Hitchin and Donaldson \cite{Hi87,Do87} constructed the moduli space of Higgs bundles over a Riemann surface $\Sigma$ and proved it to be homeomorphic to the moduli space of representations of (a central extension of) the fundamental group of $\Sigma$. Thanks to the existence theorem for twisted harmonic maps proved by Corlette \cite{Co88}, Simpson \cite{Si92,Si94} was able to extend the results to higher dimensional projective manifolds $X$: The harmonic metric constructed by Corlette gives a homeomorphism between the moduli space $\MB(X, G)$ of reductive representations of $\Gamma = \pi_1(X, x_0)$ into $G$ and the moduli space $\MDol(X, G)$ of $G$-Higgs bundles (which are assumed polystable and with some vanishing of the Chern classes). Our purpose in this paper is to study the infinitesimal behavior of the harmonic mapping with respect to the parameter $\rho \in \Hom(\Gamma, G)$; we apply this analysis to the infinitesimal study of the energy functional, which is defined on $\MDol(X, G)$ as the squared $L^2$-norm of the Higgs field $\theta$, and has so far been used intensively to study the topology of the moduli spaces in the case of a Riemann surface $\Sigma$ (cfr. \cite{Hi87}, \cite{BrGPGo06} and the references therein).

Let $M$ be a closed, orientable Riemannian manifold and $\Gamma = \pi_1(M, x_0)$; the manifold will be denoted by $X$ if we further suppose it to be K\"ahler. Let $G = \G(\R)$ be a reductive linear group, $K < G$ be maximal compact, and $N = G/K$. If $\rho_0 \colon \Gamma \to G$ is a representation, we shall identify metrics on $(\calV, D) = \big((\tM \times \R^r)/\Gamma, \de\big)$ (where $\de$ denotes the usual flat derivation) with $\rho_0$-equivariant maps $f \colon \tM \to N$, where $\tM\to M$ is the universal cover. Then, $\de f$ naturally identifies with a $\g$-valued 1-form $\beta$, such that $D = \Dcan + \beta$ is the decomposition into metric and self-adjoint parts (cfr. Proposition \ref{prop:phls}). Recall that $f$ is harmonic if and only if $\Dcan^*\beta = 0$, and in this case, if $M = X$ is K\"ahler, then $\beta = \theta+\theta^*$ and $(\calV, \theta)$ gives a Higgs bundle (with the holomorphic structure associated to the $(1,0)$-part of $\Dcan$).

We start by proving the existence of a continuous family of harmonic metrics, and the continuity of the energy functional (which, at reductive representations, where a harmonic metric $f$ exists, is half the squared norm of $\de f$). Fix a point $\tx_0 \in \tM$. Then Corlette's theorem \cite{Co88} grants the existence of a well defined map
$$
\scrH \colon Y \times \tM \to N,
$$
where $Y \subseteq N \times \Hom(\Gamma, G)$, such that $\scrH(n, \rho, \cdot)$ is the unique $\rho$-equivariant harmonic map with $\scrH(n, \rho, \tx_0) = n$. We then prove (cfr. Proposition \ref{prop:Hcontinuous}) that $Y$ is closed and $\scrH$ is a continuous map. Then, showing that the energy of a representation equals the energy of its semisimplification, we conclude that the energy functional is continuous on the whole of $\Hom(\Gamma, G)$ (Proposition \ref{prop:Econtinuous}).

The local study we carry through goes as follows: There are natural definitions of infinitesimal deformations of a representation $\rho_0$, induced by the group structures of $TG = \G(\R[t]/(t^2))$ and $J^2G = \G(\R[t]/(t^3))$, which can be rephrased as 1-cocycles in group cohomology. Analogously, deformations of a map $f \colon \tM \to N$ are naturally sections of the pull-back bundle $f^*TN$. We introduce the concepts of harmonic and equivariant deformations $(v, w)$ (with respect to a deformation of the representation $\rho_0$) in Definitions \ref{defn:vequivariant}, \ref{defn:vharmonic}, \ref{defn:wequivariant} and \ref{defn:wharmonic}. We investigate the existence of such deformations, aiming to give a way to construct them. This is completely done in the first order case, and we prove:
\begin{theoremintro}\label{thmintro:A}
 Denote by $c$ the 1-cocycle corresponding to a first order deformation of $\rho_0$, and by $\{c\} \in H^1(\Gamma, \g) \cong H^1(M, \Ad(\rho_0))$ the corresponding cohomology class (where $\Ad(\rho_0)$ is the local system on $M$ of fiber $\g$). Let $\omega \in \calH^1(M, \Ad(\rho_0))$ be its harmonic representative. Take any $F \colon \tM \to \g$ such that $\de F = \omega$ and that $F(\gamma \tx) = \Ad_{\rho_0(\gamma)}F(\tx) + c(\gamma)$ and project it via the natural map $N \times \g \to TN$. Then, we obtain a first order deformation $v$, which is harmonic and $(\rho_0, c)$-equivariant; all such deformations are obtained in this way.
\end{theoremintro}
Thanks to this result, we are able to express the first variation of the energy at $\rho_0$: Along the deformation determined by $\omega$, it becomes (Proposition \ref{prop:firstvariation}):
\begin{equation}\label{eqn:firstvariationintro}
\dev{E_t}{t}\Big|_{t=0} = \int_M \bigscal{\omega}{\beta} \de \Vol.
\end{equation}

In his paper \cite{Hi87}, Hitchin investigated the case of rank $n$ degree $d$ vector bundles on a Riemann surface $M = \Sigma$ with $(n, d) = 1$. This forces the moduli space to be smooth and projective; he then proves the energy functional $E$ to be a moment map for the $S^1$-action $t \cdot (\calE, \theta) = (\calE, t \theta)$, which gives a matching between critical points of $E$ and fixed points of the action (these are in turn the so-called ``complex variations of Hodge structures'', and their ubiquity makes them intensively studied - cfr. \cite{Si92} for a definition and proof of the ubiquity). Using \eqref{eqn:firstvariationintro}, we can prove the following:
\begin{theoremintro}\label{thmintro:B}
 The critical points of $E$ are exactly the representations induced by polarized complex variations of Hodge structure.
\end{theoremintro}
Here for critical points we mean those such that $\dev{E_t}{t}\big|_{t=0} = 0$ along all directions $c \in Z^1(\Gamma, \g)$; at smooth points, this coincides with the usual definition. Actually, the proof of Theorem \ref{thmintro:B} gives that the vanishing of the first derivative of the energy of $(\calV, t\theta)$ in $t = 1$ is sufficient.

The study of second order deformation is made harder by the presence of obstructions. It is well known (see \cite{GoMi87}) that the obstruction for a first order deformation of $\rho_0$ to be extended to the second order lies in the cohomology class of $[\omega, \omega]$, i.e. for unobstructed $\omega$ there must exist an $\Ad(\rho_0)$-valued 1-form $\psi$ on $M$ such that
\begin{equation}\label{eqn:obstructionintro}
 \de \psi = - [\omega, \omega].
\end{equation}
We have to ask such a $\psi$ to satisfy one more equation in order to assure the existence of a deformation of the harmonic metric to the second order, but this actually grants a little more:
\begin{theoremintro}\label{thmintro:C}
 Suppose that $G$ is a complex group. Fix a reductive representation $\rho_0 \colon \Gamma \to G$, a harmonic metric $f \colon \tM \to G/K$ and a harmonic 1-form $\omega \in \calH^1(M, \Ad(\rho_0))$ such that $\{[\omega, \omega]\} = 0 \in H^2(M, \Ad(\rho_0))$. Denote by $(\rho_0, c)$ the first order deformation of $\rho_0$ this defines. Then the following are equivalent:
 \begin{enumerate}
  \item There exists an $\Ad(\rho_0)$-valued 1-form $\psi$ satisfying both \eqref{eqn:obstructionintro} and $\de^*\psi = - \sum_j [\omega(E_j)^*, \omega(E_j)]$, where $\{E_j\}$ is a local orthonormal frame and $*$ denotes adjunction with respect to the harmonic metric;
  \item The harmonic 1-form $\omega$ is a minimum of the $L^2$-norm in its own orbit in $\calH^1(M, \Ad(\rho_0))$ under the adjoint action of $H = Z_G(\Image(\rho_0))$;
  \item There exists a map $(F, F_2) \colon \tM \to \g \times \g$ which is both equivariant and of harmonic type (cfr. Definitions \ref{defn:wequivariant}, \ref{defn:wharmonic});
  \item There exist two second order deformations $(v, w)$ and $(v', w')$ of $f$, both harmonic, one equivariant along (some second order extension of) $(\rho_0, c)$ and the other along $(\rho_0, ic)$.
 \end{enumerate}
 Furthermore, any of the points above is true for \emph{every} harmonic metric $f$ if and only if $H^0(M, \Ad(\rho_0, c))$ is a flat $\R[t]/(t^2)$-module (here, $\Ad(\rho_0, c) = \Ad(\rho_0 + tc)$ is the adjoint local system with fiber $\g \otimes \R[t]/(t^2)$).
\end{theoremintro}
If any of the conditions of the theorem is satisfied, then every second order harmonic map is obtained by projection of $(F, F_2)$, similarly to the first order picture. As in the first order case, we obtain a formula for the variation of the energy along deformations constructed by the means of the theorem:
\begin{equation}\label{eqn:secondvariationintro}
\devd{E_t}{t}\Big|_{t=0} = \int_M \bigscal{\psi}{\beta} + \big\|\omegap \big\|^2 \de \Vol,
\end{equation}
where $\omegap_{\tx}$ is the projection of $\omega_{\tx}$ on the subspace $[\p]_{\tx}$ of $\g$ consisting of selfadjoint elements.

Again, we want to make use of this result to prove the analog second order statements as those Hitchin proved on a Riemann surface. Namely, Hitchin \cite{Hi87} proves the energy functional $E$ to be a K\"ahler potential with respect to the complex structure given by $\MB(\Sigma, G)$, and also that it is a perfect Bott-Morse function. In \cite{Hi92}, then, he gives a formula for the eigenvalues of the Hessian of $E$ at a fixed point as a function of the eigenvalues of the infinitesimal generator of the action of $S^1$ which allow for example \cite{BrGPGo03} and \cite{GPGoMR13} to study the topology of $\MB(\Sigma, G)$ for some classes of group $G$. In the lines of the former result, we prove:
\begin{theoremintro}\label{thmintro:D}
 Let $G$ be a complex group. At the smooth points of $\MB(M, G)$, the energy functional $E$ is a Kähler potential for the Betti complex structure on $\MB(M, G)$.
\end{theoremintro}
A similar plurisubharmonicity result for the energy functional on the Teichmüller space has been recently proved by Toledo (see \cite{To12}). This work has been the original source of our interest in the question.

In order to extend Hitchin's result regarding the eigenvalues of the Hessian of $E$ (on a K\"ahler manifold $X$), we introduce the following notation, which is akin to that in e.g. \cite{Hi87}: We set $\dot{A} = (\omegak)''$ and $\dot{\Phi} = (\omegap)'$ (where $\alpha = \alpha' + \alpha''$ is the decomposition into $(1,0)$ and $(0,1)$ parts). Furthermore, at a point corresponding to a polarized variation of Hodge structure $\rho_0$, for each $\xi \in \g$, the Lie algebra of $G$, write $\xi = \sum_p \xi^{-p,p}$ for the decomposition according to the induced variation of Hodge structures of weight 0 on $\tX \times \g$, so that the infinitesimal generator of the circle action acts on $\xi^{-p,p}$ with weight $ip$.

\begin{theoremintro}\label{thmintro:E}
 Suppose that $\rho_0$ is induced by a polarized complex variation of Hodge structure. Then, with the above notations, the second derivative of the energy along a direction $\omega$ can be written as
 $$
  \devd{E_t}{t}\Big|_{t=0} = 2\int_X \sum_p \Big(-p \big\| \dot{A}^{-p,p} \big\|^2 + (1-p) \big\| \dot{\Phi}^{-p,p} \big\|^2\Big)\de \Vol.
 $$
\end{theoremintro}
\begin{cor*}
 If we assume further that $\omega$ takes values in $\g_0$, then the expression simplifies in terms of the weight 1 $(P, Q)$ Deligne-Hodge structure on $H^1(M, \Ad(\rho_0))$ as:
 $$
  \devd{E_t}{t}\Big|_{t=0} = 2 \int_X \sum_{P \textnormal{ even}} P \big\| \omega^{(P,Q)} \big\|^2.
 $$
 In particular, if $\rho_0$ is of Hermitian symmetric type, then the Hessian is semi-positive definite, and the vanishing directions are exactly those that remain complex variations of Hodge structure to the first order.
\end{cor*}

\subsection*{Organization of the paper}
 In Section \ref{sec:phls}, we introduce the notion of ``polarized harmonic local systems'' as local systems underlying harmonic bundles with a compatible involution. We prove a number of results about them, which will be needed in the following sections. Although we only apply such results in a specific class of examples (the pull-back of the ``adjoint ones'' on symmetric spaces), we state them in general. In Section \ref{sec:universal} we construct the universal twisted harmonic mapping $\scrH$, then prove its continuity and the one of the energy functional $E$.
 
 The infinitesimal study begins in Section \ref{sec:firstorder}, where the concepts of first order deformations are introduced, and Theorem \ref{thmintro:A} is proved. These results are then applied in Section \ref{sec:firstvariation}, in which we prove the formula for the first variation of the energy and use it to obtain Theorem \ref{thmintro:B}. In Section \ref{sec:secondorder} we introduce all the necessary definitions regarding second order deformations, discuss the action of $H$ on $H^1(M, \Ad(\rho_0))$ and relate it with the existence of a pair $(F, F_2)$. This existence implies that of a second order deformations; in Example \ref{ex:obstruction} we present some instances where a second order harmonic and equivariant deformation cannot exist. The best part of the section is devoted to the proof of Theorem \ref{thmintro:C}. Finally, in Section \ref{sec:secondvariation}, we prove \eqref{eqn:secondvariationintro} which we then exploit to give proofs of Theorems \ref{thmintro:D} and \ref{thmintro:E}.

\subsection*{Acknowledgements}
This work is based on the author's Ph.D. thesis at Université Joseph Fourier (Grenoble), which is publicly available at \cite{thesis}. The author would like to thank his advisor, Philippe Eyssidieux, for introducing him to the subject and for his invaluable support; the referees of my Ph.D. thesis, Olivier Biquard and Domingo Toledo, for their useful comments; Pierre Py and Alessandro Carlotto for useful comments and discussions.

\section{Polarized harmonic local systems}\label{sec:phls}
 \begin{defn}[\cite{Si92}]
  Let $M$ be a connected Riemannian manifold, $\tM \to M$ its universal cover. A \emph{harmonic bundle} is a real flat vector bundle $(\calV, D)$ of rank $n$ with a metric $h$ such that the associated map $f \colon \tM \to \GL(n, \R)/O(n)$ is harmonic.
 \end{defn}
 Here, the map $f$ is defined from $h$ by choosing a base point $x_0 \in M$ and an isomorphism $\calV \cong \tM \times_\Gamma \R^n$, with $\Gamma = \pi_1(M, x_0)$, so as to  identify sections $v, w$ with $\Gamma$-equivariant maps $\tM \to \R^n$. Then let $s \colon \tM \to G$ be such that:
 \begin{equation}\label{eqn:defnmetric}
 h(v, w)_{\tx} = \bigscal{s(\tx)^{-1} \cdot v}{s(\tx)^{-1} \cdot w}_{\std}
 \end{equation}
 where the scalar product $\scal{\cdot}{\cdot}_{\std}$ is the standard scalar product on $\R^n$. Although the definition on $s$ involves choices, the composition $f \colon \tM \xrightarrow{s} G \to G/K$ is well defined, and its harmonicity is independent of the choices of $x_0$ and of the isomorphism.
 \begin{defn}
  A (real, even) \emph{polarized harmonic local system} (phls for short) is a triple $(\V, \sigma, S)$ such that $\V$ is the local system of parallel sections of a flat bundle $(\calV, D)$, $\sigma \colon \calV \to \calV$ is an $\R$-linear involution, $S$ is a flat symmetric non-degenerate quadratic form, which is positive definite on the $+1$-eigenspace $\calV^+$ of $\sigma$ and negative definite on the $-1$-eigenspace $\calV^-$ and such that the positive-definite metric defined by
  $$
  h(v, w) = S\big(v, \sigma(w)\big)
  $$
  makes $(\calV, D)$ a harmonic bundle.
 \end{defn}

 \begin{remark}
  In both definitions, one can define complex objects by considering hermitian quadratic forms, $\C$-vector bundles, $\C$-linear involutions, etc.; also, we can consider real odd polarized harmonic local systems by considering a symplectic form $Q$ instead of a symmetric one (then it will induce a hermitian form on the complexification, hence, together with $\sigma$, a metric therein, which may be asked to be harmonic). In \cite{Si92}, only the complex setting is analyzed, hence the definitions are different from the ones above. In this section, we will always develop the theory for a real even phls, the straightforward adaptations to the remaining cases are left to the reader.
 \end{remark}
 Given a real polarized harmonic local system, its complexification bears a corresponding complex structure; furthermore, tensor products and duals (hence, endomorphisms) are defined naturally.
 
 \begin{defn}[Cfr. e.g. \cite{BuRa90}]
  Let $N = G/K$ be a Riemannian symmetric space of the non-compact type, denote by $\g$ the Lie algebra of $G$, by $\k$ that of $K$ and write $\g = \k \oplus \p$ for the symmetric decomposition. The Maurer-Cartan form $\beta_N \in \calA^1_N(\g)$ is the right inverse of $\piTN \colon N \times \g \to TN$ defined by
  $$
  \piTN(n, \xi) = \dev{}{t} \Big( \exp(t \xi) \cdot n\Big)\Big|_{t=0}.
  $$
 \end{defn}
 This 1-form gives an isomorphism at every point $\beta_{N,n} \colon T_nN \cong [\p]_n = \Ad_n(\p)$.

 \begin{defn}
  Let $(\calV, \sigma, S)$ be a polarized harmonic local system on $M$. We define the \emph{canonical connection} $\Dcan$ as the metric part of the flat connection $D$, and write
  $$
  D = \Dcan + \beta,
  $$
  so that $\beta$ is a 1-form on $M$ taking values in the selfadjoint part of $\End(\calV)$. We denote by $\de$, $\dcan$ the exterior differential operators determined by the connections $D$, $\Dcan$, respectively and by $\ncan_X$ the covariant derivation by $\Dcan$ along a vector field $X$.
 \end{defn}
 
 The rest of this section is devoted to proving the following facts, which we regroup in a proposition:
 \begin{prop}\label{prop:phls}
  Let $(\V, \sigma, S)$ be a polarized harmonic local system on the flat bundle $(\calV, D)$. Then:
  \begin{enumerate}
   \item \label{item:beta} The pull-back of $\beta$ to $\tM$ coincides with the pull-back of $\beta_N$ through the metric $f \colon \tM \to N = G/K$, where $G/K$ is any totally geodesic subspace of $\GL(n,\R)/O(n)$ in which $f$ takes values. In particular, it satisfies the ``Maurer-Cartan equation'':
   \begin{equation}\label{eqn:maurercartan}
   \de \beta = [\beta, \beta].
   \end{equation}
   \item \label{item:Dcan} The canonical connection $\Dcan$ commutes with $\sigma$, so that, for every section $v$ of $\calV$, writing $v^+$ and $v^-$ for its projections on $\calV^+$ and $\calV^-$, respectively,
   $$
   \Dcan(v^+) = \Dcan(v)^+ \quad \text{and} \quad \Dcan(v^-) = \Dcan(v)^-.
   $$
   In particular, since $\beta$ excanges $\calV^+$ and $\calV^-$,
   $$
   \Dcan(v) = \big(D v^+\big)^+ + \big(D v^-\big)^-.
   $$
   \item \label{item:codifferential} Let $\alpha$ be a $\calV$-valued 1-form. Suppose that $M$ be compact and orientable. Then the codifferential $\de^*\alpha$ may be computed (in terms of a local orthonormal frame $\{E_j\}$ of $M$) as
   \begin{equation}\label{eqn:codifferential}
   \de^*\alpha = \dcan^* \alpha + \sum_{j} \beta(E_j) \cdot \alpha(E_j) = -\sum_j \ncan_{E_j} \alpha(E_j) - \beta(E_j) \cdot \alpha(E_j).
   \end{equation}
   \item \label{item:laplacian} Let $v$ be a section of $\calV$. Then the Laplacian $\Delta v = \de^*\de v$ can be computed in terms of $\Delta^\can v = \dcan^* \dcan v = -\sum_j(\ncan_{E_j} \ncan_{E_j} v)$ and a local orthonormal frame $\{E_j\}$ as
   \begin{equation}\label{eqn:laplacian}
   \Delta v = J(v) \overset{\textnormal{def}}{=} \Delta^{\can} v + \sum_{j} \beta(E_j) \cdot \big(\beta(E_j) \cdot v\big)
   \end{equation}
   (the operator $J$ will be called the \emph{Jacobi operator}).
   \item \label{item:globalsections} Denote by $V$ the vector space of global sections of $\V$ (i.e. flat global sections of $\calV$). Then $\sigma$ leaves $V$ invariant, so that we can write $V = V^+ \oplus V^-$.
  \end{enumerate}
 \end{prop}
 \begin{notation}\label{not:G0}
  In the following, we will fix a base point $x_0 \in M$ and an isomorphism $\calV \cong \tM \times_\Gamma \R^n$, so that we also have a monodromy representation $\rho \colon \Gamma = \pi_1(M, x_0) \to \GL(n, \C)$. The map $f$ is then $\Gamma$-equivariant, where $\Gamma$ acts on $N$ through $\rho$. We shall denote by $G_0$ the Zariski closure of $\Image(\rho)$ ($G_0$ is called the monodromy group); by Corlette's theorem \cite{Co88}, $G_0$ is reductive. Because of the same theorem, there is at least one harmonic metric $f_0 \colon \tM \to G_0/K_0 \subseteq \GL(n,\R)/O(n)$. We will always denote by $s_0\colon \tM \to G_0$ one of its lifts. We denote by $H = Z_G(G_0)$ the centralizer of $G_0$ in $G$. Then, by the uniqueness part of Corlette's theorem, if $f \colon \tM \to G/K$ is any harmonic metric, $f = h \cdot f_0$, for some $h \in H$. Of course, changing the metric the phls structure changes accordingly (with obvious notations, $S$ remains the same, but the involutions are related by $\calV^+ = h \cdot \calV^{+_0}$).
 \end{notation}
 
 \begin{example}\label{ex:adjoint}
  The main example we are interested in is the \emph{adjoint polarized harmonic local system}. Let $G$ be the group of real points of a reductive algebraic group, $N = G/K$ the associated symmetric space, and consider $\calV_{\ad} = N \times \g$ as a bundle on $N$, with the trivial flat connection. Put on it the following structure: Over a point $n \in N$, the involution is the Cartan involution having as $+1$-eigenspace $[\k]_n = \Ad_n(\k)$ and as $-1$-eigenspace $[\p]_n = \Ad_n(\p)$; writing $\g = \g^{ss} \oplus \a$ for a decomposition into a semisimple ideal and the center, the metric on $\a$ is just any fixed positive definite metric while on $\g^{ss}$ it is induced by taking as symmetric form $S$ the Killing form. The resulting metric $h$ corresponds to a harmonic map $N \into \GL(\g)/O(\g) \times \a^\p$. This is in fact the totally geodesic embedding corresponding to the adjoint action $G \to \GL(\g)$. In this case, $\beta = \ad(\beta_N)$ and the ``canonical connection'' corresponds, via $\piTN$, to the usual one (i.e. the Levi-Civita connection associated to any invariant metric on $N$): This will follow from Proposition \ref{prop:phls}, as the Maurer-Cartan equation \eqref{eqn:maurercartan} and the Jacobi identity for dgla's imply, for $\xi \colon N \to \g$,
  $$
  R^{\can}\xi = \big(\de - \ad(\beta)\big)^2 \xi = -\big[[\beta, \beta], \xi\big] + \big[\beta, [\beta, \xi]\big] = \frac{1}{2} \big[[\beta,\beta], \xi\big].
  $$
  Furthermore, if one writes $\tilde D = \Dcan - \ad(\beta)$, so that from \eqref{eqn:codifferential} one has $\de^* \alpha = -\trace(\tilde{D}\alpha)$, in this case the connection $\tilde{D}$ is flat, too.
  
  The main class of examples is constructed as follows: Taking any harmonic mapping to a symmetric space $f \colon \tM \to N = G/K$, we can pull-back the structure on $\tM$. If we start with a representation $\rho \colon \Gamma = \pi_1(M, x_0) \to G$ and $f$ is $\Gamma$-equivariant, we can quotient the structure to obtain a (real, even) phls on $\tM \times_\Gamma \g \to M$.
 \end{example}

\begin{example}
 The other main class of examples, when $M = X$ is a compact K\"ahler manifold, is provided by variations of Hodge structure (VHS for short, see \cite{Si92} \S 4). The complex ones give complex phls, while the real ones give even or odd real phls depending on the parity of the weight. To this aim, one simply disregards the Hodge decomposition, only considering as (the complexification of) $\calV^+$ (resp. $\calV^-$) the direct sum of $\calV^{p,q}$ for even (resp. odd) $p$. The harmonic metric, then, is induced by the period mapping.
\end{example}

 \begin{lemma}\label{lemma:keyfact}
  Let $(\V, \sigma, S)$ be a polarized local system and denote by $\g_0$ the Lie algebra of $G_0$. Consider the flat vector bundle $\calW_0 = \tM \times_\Gamma \g_0$. Then, the restrictions to $\calW_0$ of the two polarized local systems induced on $\tM \times_\Gamma \gl_n(\R)$ by $\End(\V)$ and by $f^*(\V_\ad)$ coincide.
 \end{lemma}
 \begin{proof}
  Since $f$ and $f_0$ induce the same metric on $\calW_0$, we can work with the latter. Write $\g_0 = \bigoplus \g_i \oplus \a_0$, with $\g_i$ simple Lie algebras and $\a_0$ abelian. A straightforward verification proves that the two metrics, one obtained by tensoring $f$ with its dual metric and the other one by composing $f$ with the adjoint action, coincide. Working with the metric $f_0$, we can exploit the usual uniqueness argument for the Killing form to deduce that the flat symmetric forms, which we temporarily denote by $S_\End$ and $S_\Ad$, coincide, up to some constant multiple, on each $\g_i$. This implies equality on each $\g_i$ since both involutions have as $+1$-eigenspace a compact Lie algebra, but not as $-1$-eigenspace.
 \end{proof}
 
 To prove points \eqref{item:beta} and \eqref{item:Dcan} of proposition \ref{prop:phls}, we introduce the connection $\Dpb = D - f^*\beta_N$, which by equivariance of $f$ and of $\beta_N$ descends to a connection on the bundle $\calV \to M$. We want to prove that $\Dcan = \Dpb$. Define $\alpha = \de s \cdot s^{-1} \in \calA^1_M(\g)$ (the pull-back through $s$ of the right Maurer-Cartan form $\theta_r$ on $G$) and introduce another auxiliary connection $\Dalpha$ by $\Dalpha v = Dv - \alpha \cdot v$.
 \begin{notation}\label{not:projection}
  Let $\phi \in \calA_{\tM}^p(\g)$ be any $\g$-valued $p$-form. We define $\phi^{[\p]}$ as the composition of $\phi$ with the projection to the subbundle $f^*[\p]$ of $\tM \times \g$. When necessary, we will write explicitly $\phi_{\tx}^{[\p]}$ to denote that both $\phi$ and $[\p]$ are to be considered at $\tx$ (i.e. the projection is to $[\p]_{\tx} = \Ad_{f(\tx)}\p$). The form $\phi^{[\k]}$ is defined analogously.
 \end{notation}

 \begin{lemma}\label{lemma:propphls1}
  We have $f^*\beta_N = \alphap$. Furthermore, $\Dalpha v = s \cdot (D (s^{-1}v))$, and $\Dalpha$ is a metric connection which commutes with $\sigma$.
 \end{lemma}
 \begin{proof}
  The first assertion is a consequence of the identity $\theta_r^{[\p]} = p^*\beta_N$, where $p \colon G \to G/K = N$ and $\theta_r^{[\p]}$ is the projection of $\theta_r$ onto $p^*[\p]$, as in Notation \ref{not:projection}. This identity comes from $\piTN(\theta_r(X)) = p_*X$ for all $X \in T_gG$, which by equivariance can be proved only at $g = e$, where it is obvious. The expression for $\Dalpha$ is a straightforward computation:
  $$
   s \cdot D(s^{-1} \cdot v) = Dv - s \cdot s^{-1} \cdot \de(s) \cdot s^{-1} \cdot v = Dv - \alpha\cdot v.
  $$
  By virtue of this formula, to prove that $\Dalpha$ respects $\sigma$ is equivalent to prove that $\sigma_f = s^{-1} \circ \sigma \circ s$ is flat, i.e. $D\sigma_f = 0$. To that aim, first reduce without loss of generality to $s = s_0$. Then, by $G_0$-invariance of $S$,
  $$
  S\big(v, \sigma_f(w)\big) = S\big(s v, \sigma(s w)\big) = h(s v, sw) = \bigscal{v}{w}_{\std}.
  $$
  Since both $S$ and the standard scalar product are flat, $\sigma_f$ must be, too. Finally, the fact that $\Dalpha$ is metric is immediate computing $\de h(v, w) = \de \scal{s^{-1}v}{s^{-1}w}_{\std}$.
 \end{proof}

 \begin{lemma}\label{lemma:propphls2}
  Let $\calV$ underlie a polarized harmonic local system $(\V, \sigma, S)$. Then the induced decomposition $\End(\calV) = \End(\calV)^+ \oplus \End(\calV)^-$ coincides with the decomposition in anti-selfadjoint and selfadjoint endomorphisms, respectively.
 \end{lemma}
 \begin{proof}
  On the subbundle $\calW_0$ this follows from Lemma \ref{lemma:keyfact} and \eqref{eqn:defnmetric}, since $\calW^{+_0} = f_0^*[\gl_n(\R)^+]$, and $\gl_n(\R)^+ \oplus \gl_n(\R)^-$ is the decomposition of anti-symmetric and symmetric matrices. Now if for example $A \in \End(\calV)^+$, then as in Notation \ref{not:G0}, $h^{-1}Ah \in \End(\calV)^{+_0}$, and one reduces to the previous case.
 \end{proof}

 Thanks to Lemma \ref{lemma:propphls2}, an endomorphism commutes with $\sigma$ if and only if it is anti-selfadjoint; since by Lemma \ref{lemma:propphls1} we have $\Dpb = \Dalpha + \alphak$, then, $\Dpb$ both commutes with $\sigma$ and is metric. Furthermore, since $f^*\beta_N$ takes values in $f^*[\p]$, that is, the selfadjoint part, the decompositions $D = \Dcan + \beta = \Dpb + f^*\beta_N$ must coincide. This gives point \eqref{item:beta} of Proposition \ref{prop:phls}. To prove point \eqref{item:Dcan}, since $\beta$ anti-commutes with $\sigma$, we only need to prove that $\Dcan = \Dpb$ commutes with $\sigma$. But again, $\Dpb = \Dalpha + \alphak$, and both commute with $\sigma$.
 
 Remark also that the Maurer-Cartan equation \eqref{eqn:maurercartan} (which is proved for example in \cite{BuRa90}, Chapter 1), follows easily form the usual Maurer-Cartan equation for Lie groups, which in our notations implies $\de \alpha = \frac{1}{2}[\alpha, \alpha]$:
 $$
 \dcan \alpha = \de \alpha - [\alphap, \alpha] = \frac{1}{2} [\alpha, \alpha] - [\alphap, \alpha] = \frac{1}{2}[\alphak, \alphak] - \frac{1}{2} [\alphap, \alphap].
 $$
 This implies that $\de \beta - [\beta, \beta] = \dcan \beta = \dcan(\alphap) = 0$, since all terms on the right hand side take values in $[\k]$.
 
 The formula for the codifferential (point \eqref{item:codifferential} of Proposition \ref{prop:phls}) follows easily from the formula for the codifferential of a metric connection (see \cite{EeLe83}, (1.20)) and selfadjointness of $\beta$: Locally around a point, let $\{E_j\}$ be an orthonormal frame. Then:
 \begin{align*}
 \int_M \scal{\de^*\alpha}{v} \de \Vol &= \int_M \scal{\alpha}{\dcan v + \beta\cdot v} \de \Vol\\
  &= \int_M \sum_{j} \Big(\bigscal{-\ncan_{E_j} \alpha(E_j)}{v} + \bigscal{\beta(E_j) \cdot \alpha(E_j)}{v} \Big)\de \Vol
 \end{align*}
 (here we have abused notation since $E_j$ is only locally defined; it is to be meant that one integrates the function that is given locally around every point in such a way). The formula for $\de^*$ follows.
 
 To obtain the formula for the Laplacian (point \eqref{item:laplacian} of Proposition \ref{prop:phls}), a straightforward computation using \eqref{eqn:codifferential} and the functoriality of $\Dcan$ with respect to tensor products and duals gives:
 $$
 \Delta v = \Delta^{\can} v + \sum_{j} \beta(E_j) \cdot \big(\beta(E_j) \cdot v\big) - \sum_j \ncan_{E_j} \big(\beta(E_j)\big) \cdot v.
 $$
 We claim that the vanishing of the last term is equivalent to $f$ being harmonic: Indeed, recall that a map $f \colon \tM \to N$ is harmonic if and only if its tension field $\tau(f) = \sum_j \nN_{E_j} \de f(E_j)$ vanishes, where $\nN$ is the pull-back of the Levi-Civita connection on $N$ and $E_j$ is an orthonormal frame of $\tM$. 
 Now Example \ref{ex:adjoint} implies that $\beta_N \circ \nN = \ncan \circ \beta_N$, so that, since $\beta_N$ is injective and $\beta_N \circ \de f$ is the pull-back of $\beta$ to $\tM$, $f$ is harmonic if and only if $\sum_j \ncan_{E_j} \beta(E_j) = 0$.
 
 Finally, point \eqref{item:globalsections} of Proposition \ref{prop:phls} is an integration by parts: Taking a $v \in V$, and denoting by $v^+$ its projection on $\calV^+$, it suffices to prove that $D(v^+) = 0$. By \eqref{eqn:laplacian}, since $Dv = 0$, we have:
 $$
 0 = \int_M \bigscal{Jv}{v} \de \Vol = \int_M \big\| \Dcan v\big\|^2 \de \Vol + \int_M \big\| \beta \cdot v \big\|^2 \de \Vol
 $$
 Now both terms are non-negative, hence they must vanish. But since $\Dcan(v^+) = \Dcan(v)^+$ and $\beta \cdot (v^+) = (\beta \cdot v)^-$, these quantities must vanish, as well, and so must their sum $D(v^+)$. This concludes the proof of Proposition \ref{prop:phls}.
 
 Let us now introduce the main object to which our techniques will be applied in the next sections. Recall that if $\Gamma$ is a finitely generated group and $G$ an algebraic group, one has the representation space
 $$
 \scrR(\Gamma, G) = \Hom(\Gamma, G).
 $$
 This is actually an algebraic variety (a subvariety of $G^r$, where $r$ is the cardinality of a set of generators of $\Gamma$); the group $G$ acts on it by conjugation, and one can construct the moduli space of representations as the GIT quotient $\M(\Gamma, G) = \scrR(\Gamma, G) \GIT G$. When $\Gamma = \pi_1(M, x_0)$ this goes under the name of the ``Betti'' moduli space (cfr. \cite{Si94} \S 6), $\MB(M, G)$. We can define the ``energy functional'' $E \colon \scrR(\Gamma, G) \to \R$ by:
 \begin{equation}\label{eqn:defnenergy}
 E(\rho) = \inf \bigg\{ E(f) = \frac{1}{2} \int_M \big\| \de f\big\|^2 \de \Vol \ \bigg| \ f \colon \tM \to N \text{ is smooth and } \rho\text{-equivariant} \bigg\}.
 \end{equation}
 This is actually invariant under the conjugation action, so it descends to a functional on $\MB(M, G)$. We conclude the section by analyzing two special cases for $M$.
 
 \begin{example}
  If $M = S^1$, then harmonic mappings from $\tM$ are geodesics. Letting $g = \rho(1)$, the existence of $g$-equivariant geodesics (i.e. elements realizing the minimum in \eqref{eqn:defnenergy}) is then equivalent to $g$ being semisimple. One can see easily that in this case the energy is simply the square of the ``translation length'' (see \cite{BrHa99}):
  \begin{equation}\label{eqn:translationlength}
  E(g) = L(g)^2 = \inf_{y \in N} \dist\big( y, g \cdot y\big)^2
  \end{equation}
  (the $\geq$ inequality is given by Cauchy-Schwarz, the other one is an approximation argument starting by considering the unique geodesic arc connecting $y$ and $g \cdot y$). The proof of this fact also implies that the infimum in \eqref{eqn:translationlength} exists if, and only if, $g$ is semisimple (this fact is true in a much more general setting, see e.g. \cite{Pa11}).
 \end{example}

 \begin{example}\label{ex:kahler}
  Suppose that $M = X$ is a K\"ahler manifold. Then, there is a correspondence between harmonic bundles and some polystable Higgs bundles (\cite{Si92}, Theorem 1), which in our notations is as follows: The Higgs bundle $(\calV, \theta)$ is such that $\beta = \theta + \theta^*$ is the decomposition into $(1,0)$ and $(0,1)$ parts, and the holomorphic structure $(\calV, \bar\partial)$ is given by $\dcan = \partial + \bar\partial$. Furthermore, for harmonic bundles (hence, for phls), we have the generalized K\"ahler identities (see \cite{Si92}). Two main consequences we will be interested in are that a form is harmonic if and only if it is a zero of $\Delta' = D' D'^* + D'^*D'$, and that the pull-back of a harmonic $\calV$-valued 1-form is again harmonic. If we suppose further that $X$ be a smooth projective variety, then this correspondence gives a homeomorphism of moduli spaces between $\MB(X, G)$ and $\MDol(X, G)$, which is the moduli space of appropriate Higgs bundles; on the latter, the energy functional is the $L^2$-norm of $\theta$, hence it is continuous. In Section \ref{sec:universal} we will prove $E$ to be continuous on the whole of $\RB(M, G)$ for every Riemannian manifold $M$.
 \end{example}

 \section{The universal twisted harmonic map}\label{sec:universal}
 
 \begin{defn}
  Fix a base point $x_0$ of $M$, and let $\tx_0 \in \tM$ be a preimage. Let $\Gamma = \pi_1(M, x_0)$ and denote by $Y$ the subset of $N \times \scrR(\Gamma, G)$ given by the points $(n, \rho)$ such that there exists a $\rho$-equivariant harmonic map $f$ satisfying $f(\tx_0) = n$. Define $\scrH \colon Y \times \tM \to N$ the universal map obtained by gluing the (unique) maps above, so that $\scrH(n, \rho, \cdot)$ is $\rho$-equivariant, harmonic and $\scrH(n, \rho, \tx_0) = n$.
 \end{defn}
 By Corlette's theorem \cite{Co88}, the projection of $Y$ on the second coordinate is $\scrR(\Gamma, G)^{ss}$, the set of reductive (also called semisimple) representations. Since harmonic maps all have the same energy and are minimizers (of the expression in \eqref{eqn:defnenergy} defining $E$), for $\rho$ in $\scrR(\Gamma, G)^{ss}$ we have $E(\rho) = E(\scrH(n, \rho, \cdot))$, for any $(n, \rho) \in Y$.
 \begin{lemma}\label{lemma:semicontinuity}
  Let $\rho_t \colon \Gamma \to G$ be a smooth family of representations, for $t$ in some smooth parameter space $T \ni 0$, and $f \colon \tM \to N$ a $\rho_0$-equivariant map. Then we can always find a smooth family $f_t \colon \tM \to N$ of $\rho_t$-equivariant maps such that $f_0 = f$. In particular, the energy functional is upper semi-continuous on the whole of $\scrR(\Gamma, G)$.
 \end{lemma}
 \begin{proof}
  Maps $f_t$ correspond to metrics on the family of bundles $\calV_t = \tilde M \times_{\rho_t} \R^n$. These bundles are trivialized over common open subsets $\{U\}$, chosen independently of $t$; fix a family of local trivializations $\varphi_t^U$, smooth in $t$. The metric $f_0$ induces metrics on $\R^n$ on any local chart through $\varphi_0^U$, with a compatibility relation between charts. Composing them with $\varphi_t^U$ gives the desired family of metrics on $\calV_t$. The semi-continuity follows easily: If $f^n \colon \tM \to N$ is a minimizing sequence for $E(\rho_0)$, we deform each $f^n$ to $f_t^n$ as above. Then $f_t^n$ converges to $f^n$ in $W_{\textnormal{loc}}^{1,2}$ as $t \to 0$, thus $E(f_t^n)$ converges to $E(f^n)$. Hence
  $$
   E(\rho_0) = \lim_n E(f^n) \geq \lim_n E(f_t^n) - \varepsilon(t) \geq E(\rho_t) - \varepsilon(t), \quad \varepsilon(t) \xrightarrow{t \to 0} 0.
  $$
 \end{proof}

 \begin{prop}\label{prop:Hcontinuous}
  The subset $Y \subseteq N \times \scrR(\Gamma, G)$ is closed. The universal harmonic mapping $\scrH \colon Y \times \tM \to N$ is continuous.
 \end{prop}
 \begin{proof}
  Start from a converging sequence $Y \ni (n_t, \rho_t) \to (n_\infty, \rho_\infty)$, and fix $f_t \colon \tM \to N$, which are $\rho_t$-equivariant and such that $f_t(\tx_0) = n_t$. By Lemma \ref{lemma:semicontinuity}, the energy of $\{f_t\}$ is bounded. We can apply \cite{Li99}, Theorem A and Section 5, to deduce that on any compact subset $K \subseteq \tM$ the restrictions $f_t|_K$ are Lipschitz maps, with a uniform Lipschitz constant $L$. Together with the convergence of $f_t(\tx_0)$, we obtain a uniform bound for $f_t$ on $K$. We can then apply the $W^{2,p}$-estimates (cfr. \cite{GiTr77}, Theorem 9.11) to the semi-linear second order elliptic equation of the harmonic maps (cfr. \cite{EeSa64}, (5)), to get a uniform bound on the $W^{2,p}$-norm of $f_t$. A ``bootstrap'' argument then gives uniform bounds in every $W^{k,p}$-norm. Then, Sobolev embedding and Arzelà-Ascoli theorem give a subsequence converging in $\calC^2$ to some limit smooth map $f_\infty$. This is automatically $\rho_\infty$-invariant, and satisfies $f_\infty(\tx_0) = n_\infty$; furthermore, it must satisfy the harmonic map equation, hence it is harmonic, and $(n_\infty, \rho_\infty) \in Y$. Indeed, by uniqueness of such an harmonic map, the whole sequence $f_t$ converges to $f_\infty$; this allows us to conclude that $\scrH$ is continuous.
 \end{proof}

 Recall that, given any representation $\rho \colon \Gamma \to G$, there is a ``semisimplification'' $\rho^{ss} \colon \Gamma \to G$ defined as the graded associated to any composition series. Slightly abusing terms, we will call semisimplification of $\rho$ any point in the unique closed orbit inside the closure of the orbit $G \cdot \rho$ for the action of $G$ on $\scrR(\Gamma, G)$ by conjugation.
 
 \begin{lemma}\label{lemma:energysemisimplification}
  Let $\rho \colon \Gamma \to G$ be any representation. Then $E(\rho) = E(\rho^{ss})$.
 \end{lemma}
 \begin{proof}
  First remark that, thanks to the proof of Corlette's theorem \cite{Co88}, we can construct a family $f_n$ of $\rho$-equivariant maps such that $E(f_n)$ converges to $E(\rho)$ and also the $L^p$-norms of the first two derivatives of $f_n$ are bounded. Indeed, start by any minimizing sequence $\hat{f}^n$, and denote $\hat{f}_t^n$ the metric constructed via the heat flow starting from $\hat{f}^n$. In Corlette's notations, if we call $\Phi_t^n$ the moment map associated to $\hat{f}_t^n$ (i.e. its tension field), we have $\|\Phi_t^n\|_{L^\infty} \xrightarrow{t\to\infty}0$; thanks to the estimates in Corlette's paper, the $W^{1,p}$ norm of $\beta_t^n = (\de \hat{f}_t^n \cdot (\hat{f}_t^n)^{-1})^{[\p]}$ (in his notations: $\theta_t^n$) are, for $t = t(n)$ big enough, bounded by a constant depending only on $E(\hat{f}_t^n)$, hence by a constant, since both $\hat{f}^n$ and the heat flow are energy-decreasing. Defining $f_n = \hat{f}_{t(n)}^n$ gives the desired sequence.
  
  Secondly, let $g_n$ be such that $f_n(\tx_0) = g_nK$, and define $\tilde \rho_n = g_n^{-1}\rho g_n$. We want to prove that $\tilde \rho_n$ subconverges to some $\tilde \rho_\infty$. By properness of $G \to G/K$, it suffices to prove that $\tilde \rho_n(\gamma)K$ remains at bounded distance from $eK$ for all $\gamma \in \Gamma$. This follows from the Lipschitz estimates on $\tilde f_n$ (coming from those on $f_n$), since:
  $$
   \dist\big(eK, \tilde \rho_n(\gamma) K\big) = \dist\big(\tilde f_n(\tx_0), \tilde f_n(\gamma\tx_0)\big) \leq L(\gamma) \cdot \dist(\tx_0, \gamma \tx_0) = C(\gamma).
  $$
  
  Lastly, we want to prove that $\tilde \rho_\infty = \rho^{ss}$ and that $E(\rho) = E(\tilde \rho_\infty)$. It is clear that $\tilde \rho_\infty$ is in the closure of the orbit of $\rho$. Furthermore, arguing as in Proposition \ref{prop:Hcontinuous}, $\tilde{f_n}$ converges in $W^{1,p}$ to some $\tilde f_\infty$, which is at least $\calC^1$. In fact, it is harmonic, since it minimizes the energy: This follows from $E(f_n) = E(\tilde f_n) \xrightarrow{n \to \infty} E(\tilde f_\infty)$, together with the following chain of inequalities (here we use also Lemma \ref{lemma:semicontinuity} for the second inequality):
  $$
   E(\tilde f_\infty) \geq E(\tilde\rho_\infty) \geq \limsup E(\tilde \rho_n) = E(\rho) = \lim E(f_n) = E(\tilde f_\infty).
  $$
  Remark that this also shows that $E(\rho) = E(\tilde \rho_\infty)$, which concludes the proof.
 \end{proof}
 \begin{prop}\label{prop:Econtinuous}
  The energy functional is continuous on the whole of $\scrR(\Gamma, G)$.
 \end{prop}
 \begin{proof}
  Let $\rho_t \to \rho_\infty$ be a converging sequence. Firstly, if we assume that $\rho_t$ and $\rho_\infty$ are semisimple and that there exist a converging family $n_t \to n_\infty$ such that $(n_t, \rho_t) \in Y$ we can conclude at once by the proof of Proposition \ref{prop:Hcontinuous}, since $W^{1,2}$-convergence implies convergence of the energies.
  
  Secondly, suppose only that $\rho_t$ and $\rho_\infty$ are semisimple, and let $n_\infty$ be such that $(n_\infty, \rho_\infty) \in Y$. Then there are $g_t \in G$ such that $\tilde \rho_t = g_t \rho_t g_t^{-1}$ verify $(n_\infty, \tilde \rho_\infty) \in Y$. Proceeding as in Lemma \ref{lemma:energysemisimplification}, a subsequence of this converges to some $\tilde \rho_\infty$, which is semisimple. Thus $\tilde \rho_\infty$ is conjugated to $\rho_\infty$, since the quotient of semisimple representations by conjugation is a Hausdorff space. Then $E(\rho_t) \to E(\rho_\infty)$ follows from the first point.
  
  Now proceed to the general case. Without loss of generality, we may suppose $G = \GL(n, \C)$. Denote by $\rho_t^{ss}$ and $\rho_\infty^{ss}$ the corresponding semisimplifications, and by $\{\rho_t^{ss}\}$ and $\{\rho_\infty^{ss}\}$ the closed points of $\MB(M, G)$ they represent. Since the functions on $\MB(M, G)$ are generated by the traces, and $\trace(\rho_t) = \trace(\rho_t^{ss})$ converges to $\trace(\rho_\infty) = \trace(\rho_\infty^{ss})$, we have convergence of the closed points in $\MB(M, G)$. By the Kempf--Ness theorem \cite{KeNe78}, $\MB(M, G) \cong \mu^{-1}(0)/K$, with $\mu^{-1}(0) \subseteq \scrR(\Gamma, G)^{ss}$. Lifting the closed points to some $\tilde \rho_t^{ss}$, which must then be conjugated to $\rho_t^{ss}$, by properness of $\mu^{-1}(0) \to \mu^{-1}(0)/K$, there is a subsequence converging to some $\tilde \rho_\infty^{ss}$, which must be conjugated to $\rho_\infty^{ss}$, as well. Then we conclude thanks to the second part and Lemma \ref{lemma:energysemisimplification}.
 \end{proof}

 Remark that, on the locus of the Zariski dense representations, these results are trivially implied by the following:
 \begin{prop}[\cite{Co91}, Proposition 2.3]
  Let $R$ be an irreducible component of $\scrR(\Gamma, G)$, and give its smooth part $R^{sm}$ the $\Cinfty$ structure induced by the reduced structure on $R$. Denote by $U$ the (possibly empty) open subset of $R^{sm}$ such that $\Image(\rho) \subset G$ is Zariski dense. Then, the restriction of $\scrH$ to $Y \cap (N \times U)$ is smooth, hence the same is true for the energy functional on $U$.
 \end{prop}

 \section{First order deformations}\label{sec:firstorder}
 \begin{defn}
  A \emph{first order deformation} $v$ of a map $f \colon \tM \to N$ is a smooth section of the bundle $f^*TN$.
 \end{defn}
 We will often denote $v$ by $\dev{f_t}{t}\big|_{t=0}$. Clearly, when $f_t$ is defined and smooth for a real parameter $t$, this gives a class of examples of first order deformations $v$; but interpreting $t$ as a formal parameter (seeing $N$ as the set of real points of an algebraic variety, and $TN$ as the set of $\R[t]/(t^2)$-points of the same variety) we can work in greater generality (e.g. allowing obstructions).
 \begin{defn}
  A first order deformation $\rhotfirst$ of $\rho_0 \colon \Gamma \to G$ is a representation $\rhotfirst \colon \Gamma \to TG$ projecting to $\rho_0$ via $TG \to G$.
 \end{defn}
 Here, $TG$ is given the group structure induced by $\G(\R[t]/(t^2))$, where $G = \G(\R)$. Explicitly (see \cite{BrSu72}), we can write elements of $TG$ as pairs $(g, \xi)$ with $g \in G$ and $\xi \in \g$, with the product structure $(g, \xi) \cdot (h, \eta) = (gh, \xi + \Ad_g\eta)$. Then we have $\rhotfirst = (\rho_0, c)$ with $c$ a 1-cocycle of the adjoint representation, i.e. $c \colon \Gamma \to \g$ satisfies
 $$
 c(\gamma \eta) = c(\gamma) + \Ad_{\rho_0(\gamma)} c(\eta).
 $$
 Again, starting from a family $\rho_t$ of representations, we obtain its first order deformation by defining $c(\gamma) = \dev{\rho_t(\gamma)}{t}\big|_{t=0} \rho_0(\gamma)^{-1}$.
 
 We are interested in the following problem: Given a harmonic, $\rho_0$-equivariant map $f \colon \tM \to N$ and a first order deformation $\rhotfirst$ of $\rho_0$, can we describe first order deformations which remain ``harmonic'' and ``$\rhotfirst$-equivariant'' to the first order? We first have to define such terms, introducing in passing auxiliary functions $F \colon \tM \to \g$.
 \begin{defn}\label{defn:vequivariant}
  A first order deformation $v$ of $f$ is $\rhotfirst$-equivariant if, and only if, it is for the action of $TG$ on $TN$ as in \cite{BrSu72}. Explicitly, this writes
  $$
  v(\gamma \tx) = \rho(\gamma)_* v(\tx) + \piTN\big(f(\gamma \tx), c(\gamma)\big).
  $$
 \end{defn}
 \begin{defn}
  A function $F \colon \tM \to \g$ is $\rhotfirst$-equivariant if $(f, F)$ is, under the left action $TG \cong G \times \g \acts G \times \g / K \cong N \times \g$. Explicitly, this means
  $$
  F(\gamma \tx) = \Ad_{\rho_0(\gamma)} F(\tx) + c(\gamma).
  $$
 \end{defn}
 \begin{lemma}\label{lemma:jacobi}
  If $f_t \colon \tM \to N$ is a family of $\rho_t$-equivariant maps, then $\dev{f_t}{t}\big|_{t=0}$ is $\rhotfirst$-equivariant. If $F$ is a $\rhotfirst$-equivariant function, defining $v = \piTN(f, F)$ gives a $\rhotfirst$-equivariant first order deformation of $f$.
 \end{lemma}
 The proof of both statements is immediate.
 \begin{defn}[See \cite{Ma73},\cite{EeLe83}]\label{defn:vharmonic}
  The \emph{Jacobi operator} $\calJ \colon \Cinfty(f^*TN) \to \Cinfty(f^*TN)$ is defined in terms of a orthonormal local frame $\{E_j\}$ as
 \begin{align*}
  \calJ(v) \overset{\textnormal{loc}}{=} -\sum_{j} \Big( \nN_{E_j} \nN_{E_j} v + R^N\big(\de(f)(E_j), v\big) \de(f)(E_j)\Big),
 \end{align*}
 where $R^N$ is the curvature of the Levi-Civita connection on $N$. A first order deformation $v$ is said to be harmonic if $\calJ(v) = 0$.
 \end{defn}
 
 \begin{notation}\label{not:pullbackadjoint}
  Given a representation $\rho_0 \colon \Gamma \to G$ and a harmonic $\rho_0$-equivariant map $f \colon \tM \to N$, we will denote by $\calV$ the vector bundle underlying the phls induced on $M$ as in Example \ref{ex:adjoint}. The corresponding local system will be denoted by $\V = \Ad(\rho_0)$. Recall that in this case $\calV = \tM \times_\Gamma \g$, and, slightly abusing notations, we will write $\beta \in \calA_M^1(\g)$ for the 1-form induced by $\tbeta = f^*\beta_N$ (so that what we called $\beta$ in Section \ref{sec:phls} would be, in present notations, $\ad(\beta)$). For the sake of brevity, we will write $\scal{\xi}{\eta}$ for the metric $h(\xi, \eta)$. Recall that $\beta_N \circ \nN = \ncan \circ \beta_N$, where $\nN$ is the pull-back connection of the Levi-Civita connection on $f^*TN$. It follows that $\beta_N \circ \calJ = J \circ \beta_N$, where $J$ is defined as in \eqref{eqn:laplacian}. Coherently with Notation \ref{not:projection}, we may speak of projections $\xi_{\tx}^{[\p]}$, etc.
 \end{notation}

 \begin{lemma}
  Let $f_t \colon \tM \to N$ be a family of harmonic maps, varying smoothly in $t$. Then $v = \dev{f_t}{t}\big|_{t=0}$ is a harmonic first order deformation.
 \end{lemma}
 \begin{proof}
  We simply differentiate the identities $\tau(f_t) = 0$ covariantly along $t$. Write for short $\covt$ for $\nN_{\dev{}{t}}$; then, with respect to a fixed local orthonormal frame $\{E_s\}$, using the ``symmetry relations'' (cfr. \cite{doCa92}, Chap. 3, Lemma 3.4 and Chap. 4, Lemma 4.1), we have:
  \begin{align*}
    \frac{D}{\partial t} \sum_s (\nN_{E_s} \de f_t(E_s))\big|_{t=0} &=  \sum_s \frac{D}{\partial t} \nN_{E_s} \de f_t(E_s) \big|_{t=0}\\
  &= \sum_s \nN_{E_s} \frac{D}{\partial t} \de f_t(E_s) \big|_{t=0} + R^N(\de f(E_s), v) \de f(E_s)\\
  &= \sum_s \nN_{E_s} \nN_{E_s} v + R^N(\de f(E_s), v) \de f(E_s)= -\calJ(v).
  \end{align*}
 \end{proof}

 Remark that we can extend the definition of $J$ in \eqref{eqn:laplacian} to general function $\xi$ simply by \emph{defining} $\Delta^{\can}\xi = -\sum_j \ncan_{E_j} \ncan_{E_j} \xi$.
 \begin{defn}
  A function $F \colon \tM \to \g$ is called ``of harmonic type'' if $J(F) = 0$.
 \end{defn}
 Then, defining $v = \piTN(f, F)$ we can map functions of harmonic type to harmonic first order deformations: If $J(F) = 0$,
 $$
 \beta_N \big(\calJ(v)\big) = J\big(\beta_N(v)\big) = J\big(\Fp\big) = J(F)^{[\p]} = 0,
 $$
 since $J$ respects the decomposition $\tM \times \g = [\k] \oplus [\p] = \calV^+ \oplus \calV^-$.

 \begin{notation}\label{not:omega}
  Let $f$ be a $\rho_0$-equivariant harmonic map, $\rhotfirst$ a first order deformation of $\rho_0$. We denote by $\omega \in \calH^1(M, \Ad(\rho_0))$ the harmonic 1-form representing the 1-cohomology class given by $\{c\} \in H^1(\Gamma, \g) \cong H^1(M, \Ad(\rho_0))$. Keeping the same notation as in \ref{not:G0}, denote by $\h$ the Lie algebra of $H$, and remark that $\h = H^0(M, \Ad(\rho_0))$ is the space of global sections of $\Ad(\rho_0)$; hence by point \eqref{item:globalsections} of Proposition \ref{prop:phls}, it splits as a direct sum which we denote $\h = \hk \oplus \hp$, since $\h^+ = \hk = \h \cap \k$ and $\h^- = \hp = \h \cap \p$.
 \end{notation}

 The main result about first order deformations is then the following:
 \begin{theorem}\label{thm:firstorder}
  Let $M$ be a compact Riemannian manifold, $\G$ an algebraic reductive group, $G = \G(\R)$ the Lie group of its real points, $\rhotfirst = (\rho_0, c) \colon \pi_1(M) \to G$ a first-order deformation of $\rho_0$ and $f \colon \tM \to N$ a harmonic and $\rho_0$-equivariant map. Then the set of $\rhotfirst$-equivariant $F$ such that $\de F = \omega$ is non-empty; in fact, it is an affine space over $\h$, and every harmonic first order deformation $v$ is constructed as $\piTN(f, F)$. More precisely, the map:
  $$
  \left\{
  \begin{array}{r}
  F \colon \tM \to \g \ : \ \de F = \omega \text{ is harmonic }\\
  \text{and } F(\gamma \tx) = \Ad_{\rho_0(\gamma)}F(\tx) + c(\gamma)
  \end{array}
  \right\}
  \xrightarrow{\piTN}
  \left\{
  \begin{array}{c}
   v \in \Cinfty(f^*TN) \text{ harmonic}\\
   \text{and } \rhotfirst\text{-equivariant}.
  \end{array}
  \right\}
  $$
  is affine and surjective, and corresponds to the linear projection on the associated vector spaces:
  $$
  \h = H^0(M, \Ad(\rho_0)) \longrightarrow H^0(M, \Ad(\rho_0)) \cap \p = \hp.
  $$
 \end{theorem}
 
 Remark that, since the isomorphism $H^1(M, \Ad(\rho_0)) \cong H^1(\Gamma, \g)$ is induced by integration, if $F$ is of harmonic type and $\rhotfirst$-equivariant then $\int_{\tx_0}^{\gamma \tx_0} \de F = c(\gamma) + \delta(F(\tx_0))(\gamma)$, where $\delta$ denotes the codifferential of group cohomology, hence necessarily $\de F = \omega$. The fact that $\rhotfirst$-equivariant primitives $F$ of $\omega$ exist, and that they form an affine space over $\h$ is a consequence of the following lemma:
 \begin{lemma}\label{lemma:existence}
  Let $V$ be a fixed vector space of finite dimension, and $\tau \colon \Gamma \to \GL(V)$ a representation. Denote by $\V$ the associated local system and let $\phi \in Z^1(M, \V)$ be a closed 1-form (which we think of as a $\tau$-equivariant closed 1-form on $\tM$); let $z \in Z^1(\Gamma, V)$ be a 1-cocycle such that the cohomology classes of $\phi$ and $z$ correspond through the isomorphism $H^1(M, \V) \cong H^1(\Gamma, V)$. Then the set
  \begin{equation}\label{eqn:torsorF}
    \Big\{ F \colon \tM \to V \ :\ \de F = \phi \textnormal{ and } F(\gamma \tx) = \tau(\gamma)\cdot F(\tx) + z(\gamma) \Big\}
  \end{equation}
  is an affine space over $V^\Gamma = H^0(M, \V)$.
 \end{lemma}
 \begin{proof}
  Taking any $F$ such that $\phi = \de F$, by equivariance the 1-cocycle $z_F$ defined by $z_F(\gamma) = F(\gamma \tx) - \tau(\gamma) \cdot F(\tx)$ is independent of $\tx$. By hypothesis, the cocycle $\gamma \mapsto \int_{\tx_0}^{\gamma \tx_0} \phi = F(\gamma \tx_0) - F(\tx_0)$ is cohomologous to $z$; but it is also cohomologous to $z_F$, since:
  $$
   F(\gamma \tx_0) - F(\tx_0) = z_F(\gamma) + \tau(\gamma) \cdot F(\tx_0) - F(\tx_0) = z_F(\gamma) + \delta\big(F(\tx_0)\big)(\gamma).
  $$
  Now if $\de F = \de \tilde F = \phi$, the difference between $F$ and $\tilde F$ is a fixed element $v \in V$, and the difference between $z_F$ and $z_{\tilde F}$ is the coboundary $\delta(v)$. Hence we can find one $F$ such that $z_F = z$ (that is, which is $\rhotfirst$-equivariant) and the difference of any two such choices must be $\Gamma$-invariant, as claimed.
 \end{proof}
 \begin{proof}[Proof of theorem \ref{thm:firstorder}]
  The lemma gives the first part, hence the existence of at least one harmonic $\rhotfirst$-equivariant first order deformation $v$. To conclude the proof, we only need to show that the difference of any two such deformations $v$, $v'$ is in $\hp$. Equivariance implies that $\xi = \beta_N(v - v')$ is $\Ad(\rho_0)$-equivariant; harmonicity implies $J(\xi) = 0$. Since $J = \de^*\de$, an integration by parts gives $\de(\xi) = 0$. Hence, $\xi \in \h$. On the other hand, both $\beta_N(v)$ and $\beta_N(v')$ are sections of $[\p]$, so $\xi \in \hp$, as claimed. Conversely, adding an element of $\hp$ to $\beta_N(v)$ gives another harmonic $\rhotfirst$-equivariant first-order deformation, hence the space of such deformations is affine over $\hp$ (and non-empty by the existence of $F$ above), as claimed.
 \end{proof}
 \begin{remark}
  One can easily prove that if $X$, $X'$ are K\"ahler manifolds and $\varphi \colon X' \to X$ is holomorphic, then the construction of theorem \ref{thm:firstorder} is functorial under pull-back by $\varphi$ (i.e. the surjective arrow fits in a square diagram with $\varphi^*$ as vertical arrows). The only non-trivial part are harmonicity of $f \circ \varphi$, which is classic (see \cite{ABCKT96}, chapter VI and \cite{Lo99}) and of $\varphi^*\omega$ (cfr. Example \ref{ex:kahler}).
 \end{remark}
 \begin{example}
  We see readily that when $G = \C^*$ is abelian, one gets back the usual abelian cohomology and harmonic functions. In this case $D = \Dcan$ is metric, and a representation $\rho_0 \colon \Gamma \to \C^*$ decomposes into a real and a unitary factor. The logarithm of the former gives a $1$-cohomology class on $M$, whose harmonic representative may be integrated to give a harmonic map $f \colon \tM \to N \cong \R$. Then $J = \Delta$ is the usual Laplace-Beltrami operator (up to a sign), hence functions $F$ of harmonic type are just harmonic complex functions; the projection $\piTN$ of theorem \ref{thm:firstorder} simply consists in taking the real part.
 \end{example}

 \section{$\C$-VHS as critical points of the energy}\label{sec:firstvariation}
 Let $f$ and $\rhotfirst$ be as above, and $v$ a $\rhotfirst$-equivariant deformation of $f$. Remark that the function defined on $\tM$ by $\scal{\nN v}{\de f}$ is $\Gamma$-invariant: Applying $\beta_N$ everywhere, at a point $\gamma \tx$ it equals
 \begin{equation}\label{eqn:firstderivative}
 \Bigscal{\Dcan \Big(\Ad_{\rho_0(\gamma)} \beta_N(v) + c(\gamma)_{\gamma \tx}^{[\p]}\Big)}{\Ad_{\rho_0(\gamma)}\tbeta}_{\gamma \tx}.
 \end{equation}
 Now the metric is $\Gamma$-equivariant, hence the first summand reduces to $\scal{\Dcan \beta_N(v)}{\tbeta}_{\tx}$, which is exactly the value of $\scal{\nN v}{\de f}$ at $\tx$. The second summand vanishes, because $\Dcan c(\gamma)^{[\p]} = -[\tbeta, c(\gamma)^{[\k]}]$, whose scalar product with $\tbeta$ equals the scalar product of $c(\gamma)^{[\k]}$ with $\trace[\tbeta, \tbeta] = 0$. Thus we can define:
 \begin{defn}
  Let $v$ be a $\rhotfirst$-equivariant first order deformation of $f$. We define the energy of $(f, v)$ as
  $$
  E(f, v) = E(f) + t \int_M \bigscal{\nN v}{\de f} \de \Vol \in \R[t]/(t^2).
  $$
 \end{defn}
 It is easy to see that when $f_t$ is $\rho_t$-equivariant, the two definitions for $\dev{E(f_t)}{t}\big|_{t=0}$ coincide.
 \begin{prop}\label{prop:firstvariation}
  Let $v$ be a harmonic and $\rhotfirst$-equivariant deformation of $f$. Then, keeping the notations in \ref{not:omega}, we have
  $$
  \dev{E(f,v)}{t}\Big|_{t=0} = \int_M \bigscal{\omega}{\beta} \de \Vol.
  $$
 \end{prop}
 \begin{proof}
  By Theorem \ref{thm:firstorder}, we have $\Fp = \beta_N(v)$, for some $\rhotfirst$-equivariant $F \colon \tM \to \g$ such that $\de F = \omega$. By definition, then,
  $$
  \int_M \scal{\omega}{\beta} \de \Vol = \int_{\tM/\Gamma} \Big(\bigscal{\Dcan F}{\tbeta} + \bigscal{[\tbeta, F]}{\tbeta} \Big) \de \Vol.
  $$
  Now the second summand vanishes since $\trace[\tbeta,\tbeta] = 0$, as above. By orthogonality of $[\p]$ and $[\k]$, then, the first one equals $\scal{\Dcan \Fp}{\beta}$, which is what we wanted.
 \end{proof}

 \begin{example}
  When $M = X$ is a K\"ahler manifold, the expression is independent of the chosen metric in its K\"ahler class. Denoting by $\Omega$ the K\"ahler form in the notations of Example \ref{ex:kahler}, we can write the first variation of the energy as:
  $$
   \dev{E_t}{t}\Big|_{t=0} = \int_{X} \big\langle\omega \wedge *\beta\big\rangle = -\frac{1}{n!} \int_X \trace\big(\omega \wedge (\theta^* - \theta)\big) \wedge \Omega^{n-1}
  $$
  (since $*\beta = (\theta^*-\theta)\wedge \Omega^{n-1}$ takes values in the anti-selfadjoint part of $\g \otimes \C$). To prove the closeness of $\trace(\omega \wedge (\theta^* - \theta))$, one simply observes that $\de \omega = 0$ and $\de \theta = \de (\theta^*) = [\theta, \theta^*]$.
 \end{example}
 
 From now on, let $M = X$ be a K\"ahler manifold. We want to analyze the critical points of the energy functional on the representation space $\scrR(\Gamma, G)$. We start by two easy examples:
 \begin{example}\label{ex:C*}
  When $G = \C^*$, the formula of Proposition \ref{prop:firstvariation} is essentially trivial, at least for smooth families of harmonic functions $f_t$, since then $\omega = \dev{\beta_t}{t}\big|_{t=0}$ (in this case, $\beta_t = \de \log(f_t)$). Furthermore, when $M = X$ is smooth projective, the Dolbeault moduli space splits as $\MDol(X, \C^*) = \Pic^0(X) \oplus H^0(X, \Omega_X^1)$, so that the $\C^*$-action (see \cite{Si92}, \S 4) defined by $t \cdot (\calV, \theta) = (\calV, t\theta)$ has as only fixed points those with $\theta = 0$. These points are also the global minima of the energy $E(\calV, \theta) = \|\theta\|^2$, which are actually the only critical points. So, in this case, critical points, global minima and $\C$-VHS are synonymous.
 \end{example}
 \begin{example}
  When $X = \Sigma$ is a Riemann surface of genus $g \geq 2$, the energy functional has be intensively studied starting with \cite{Hi87}. In that case, it is a moment map for the circle action, hence the (smooth) critical points are exactly those induced by a $\C$-VHS. Remark that this implies the same for general smooth projective $X$ (with very ample metric): The energy can be expressed as $E(\calV, \theta) = \int_X \trace(\theta \wedge \theta^*) \wedge \Omega^{n-1}$, and the cohomology class determined by $\Omega^{n-1}$ can be taken as that of a smooth curve which is the complete intersection of $n-1$ hyperplane sections. Then Simpson (\cite{Si92}, \S 1) proves the functoriality with respect to pull-backs of the Higgs bundle associated to a representation and (loc. cit., \S 4) of the fixed points of the $\C^*$-action with respect to restrictions to a complete intersection, thus we can reduce to the case of curves.
 \end{example}

 Recall that when $\rho_0$ is induced by a $\C$-VHS, there is a harmonic map $f_0 \colon \tX \to G_0/K_0$ induced by the period mapping $\Phi_0 \colon \tX \to G_0/V_0$. Suppose that $f \colon \tX \to G/K$ is induced by this map (after some totally geodesic embedding $G_0/K_0 \subseteq G/K$). Denote by $\g_\C = \g \otimes \C$ the complexification of $\g$. Then, the vector bundle $\tilde X \times \g_\C$ has a Hodge decomposition of weight 0, $\tX \times \g_\C = \bigoplus_p [\g^{-p,p}]$. Denote by $\gamma$ the infinitesimal generator of the circle action, that is a section of the subbundle $[\v_0] \subseteq \tX \times \g$ (here $\v_0 \subseteq \k_0$ is the Lie algebra of $V_0$). It is determined by:
 \begin{equation}\label{eqn:defngamma}
 [\gamma, \xi] = \sum_p ip\xi^{-p,p}, \quad \text{where } \xi = \sum_p \xi^{-p,p}, \ \, \xi^{-p,p} \in [\g^{-p,p}].
 \end{equation}
 \begin{defn}
  A point $\rho \in \MB(X, G)$ is a \emph{critical point of the energy} if $\int_X \scal{\omega}{\beta} \de \Vol = 0$ for every $\omega \in \calH^1(X, \Ad(\rho_0))$.
 \end{defn}
 Remark that the choice of the harmonic metric is inconsistent. Also, if $\{\rho\}$ is a critical point, then the vanishing holds for every $\omega \in \calH^1(X, \Ad(\rho_0) \otimes \C)$, since $\g$ and $i\g$ are orthogonal, and more generally for every $\omega \in \calH^1(X, \tilde\g)$, where $\g \subset \tilde \g$ is a Lie subalgebra. Finally, note that if $\{\rho\}$ is a smooth point, then this definition coincides with the usual notion of critical point.

\begin{theorem}\label{thm:firstvariation}
  A point of the moduli space $\MB(X, G)$ is a critical point of the energy functional if, and only if, it is induced by a complex variation of Hodge structure. Using the Higgs bundle notation, this is actually equivalent to the unique vanishing: $\dev{E(\calV, t\theta)}{t} \big|_{t=1} = 0$.
 \end{theorem}
 \begin{proof}
  For the ``if'' part, let $\rho_0$ be induced by a $\C$-VHS, and give $\Ad(\rho_0)$ the metric induced by the period mapping. Any tangent direction to $\{\rho_0\}$ can be lifted to a $\rhotfirst$, hence giving rise to a harmonic 1-form $\omega$. Remark that since $\theta \in \calA^{1,0}([\g^{-1,1}])$ and $\theta^* \in \calA^{0,1}([\g^{1,-1}])$, \eqref{eqn:defngamma} implies
  \begin{equation}\label{eqn:betadec}
  \beta = \theta + \theta^* = i D'' \gamma - i D'\gamma = D^c \gamma.
  \end{equation}
  Thanks to Proposition \ref{prop:firstvariation}, we can then compute the variation of the energy along that direction:
  $$
  \dev{E}{t}\Big|_{t=0} = \int_X \bigscal{\omega}{D^c \gamma} \de \Vol = \int_X \bigscal{D^{c,*}\omega}{\gamma} \de \Vol = 0.
  $$
  For the ``only if'' part, consider the variation of the energy with respect to the $\C^*$-action. Actually, we will work with the action of $\R_{>0}$ only, since the energy is invariant under the circle action; so take the family of Higgs bundles $(\calV, t\theta)$ for $t \in (1-\varepsilon, 1+\varepsilon)$. The first order of this family corresponds to an element of the first cohomology group of the complex $(\calA^\bullet(\calV), D'')$; clearly, this element is represented by $\theta$ itself. Hence, the harmonic 1-form to which it corresponds (cfr. \cite{Si92}, Lemma 2.2) is $\omega = \theta + D''\eta$, for some $\eta \in \Cinfty(\calV)$. We obtain:
  $$
  0 = \dev{E(\calV, t\theta)}{t}\Big|_{t=0} = \Re \int_X \bigscal{\omega}{\theta+\theta^*}_\C \de \Vol = \Re \int_X \Big(\bigscal{\omega}{\theta}_\C + \bigscal{\bar\partial \eta}{\theta^*}_\C\Big) \de \Vol,
  $$
  where $\scal{\cdot}{\cdot}_\C$ denotes the Hermitian extension of $\scal{\cdot}{\cdot}$, so that the harmonic metric on $\tX \times \g_\C$ is its real part. Now $\int_X \scal{\omega}{\theta}_\C = \|\omega\|_{L^2}^2$, since the harmonicity of $\omega$ implies $\int \scal{\omega}{D''\eta} \de \Vol = 0$; also, $\int \scal{\bar\partial \eta}{\theta^*} = 0$, through an integration by parts, Stokes theorem and the identity $\partial \theta^* = 0$. Hence, $\omega = 0$. This implies $\theta = -D''\eta = -[\theta, \eta]$. Consider the 1-parameter group of automorphisms of $\calV$ defined by $g_t = \exp(t \eta)$; then $\Ad_{g_t}(\eta) = e^{-t}\theta$. Thus, $g_t \colon (\calV, \theta) \cong (\calV, e^t\theta)$, hence $\rho_0$ is induced by a $\C$-VHS. The last statement follows immediately.
 \end{proof}

 \section{Second order deformations}\label{sec:secondorder}
 
 \subsection{Equivariant and harmonic deformations}
 
 \begin{defn}
  A second order deformation of a map $f \colon \tM \to N$ is a pair of tangent fields along $f$ which we denote by
  \begin{equation} \label{eqn:secondorder}
   \Big( v \overset{\textnormal{not}}{=} \dev{f_t}{t}\Big|_{t=0},\ w \overset{\textnormal{not}}{=} \frac{D}{\partial t} \dev{f_t}{t}\Big|_{t=0}\Big) \in f^*\big(TN \times_N TN\big).
  \end{equation}
 \end{defn}
 Through the canonical connection, this description is equivalent to taking a section of the second jet bundle $J^2N$, which is a homogeneous space acted upon by the group $J^2G = \G(\R[t]/(t^3))$, that is in bijection with $G \times \g \times \g$ and whose group structure is, under this ``right trivialization $r$'', (cfr. \cite{Be08}, \S 23):
 \begin{equation}\label{eqn:productJ2G}
  (g, \xi, \mu) \cdot (h, \eta, \nu) = \Big(gh, \xi + \Ad_g (\eta), \mu + \Ad_g (\nu) + [\xi, \Ad_g(\eta)]\Big).
 \end{equation}
 \begin{defn}
  A second order deformation $\rhotsecond$ of $\rho_0$ (resp. of $\rhotfirst$) is a representation $\rhotsecond \colon \Gamma \to J^2G$ projecting to $\rho_0$ (resp. to $\rhotfirst$).
 \end{defn}
 Through the above trivialization, we can write $\rhotsecond = (\rho_0, c, k)$. Then one observes that $\rhotsecond$ is a deformation of $\rhotfirst = (\rho_0, c)$ if and only if $k \colon \Gamma \to \g$ makes the pair $(c, k)$ a 1-cocycle for the adjoint action of $\Gamma$ on $\g \otimes \R[t]/(t^2)$, that is, $(c + tk)(\gamma \eta) = c(\gamma) + t k(\gamma) + \gamma \cdot (c(\eta) + t k(\eta))$, where
 \begin{equation}\label{eqn:action2}
  \gamma \cdot (\xi + t \mu) = \Ad_{\rho_0(\gamma)}(\xi) + t \Big( \Ad_{\rho_0(\gamma)}(\mu) + \big[c(\gamma), \Ad_{\rho_0(\gamma)}(\xi)\big]\Big).
 \end{equation}
 \begin{lemma}\label{lemma:diagram}
  Keeping the same notations as in \ref{not:projection}, define a map $\piJN \colon N \times \g \times \g \to TN \times_N TN$ by
  $$
   \piJN(n, \xi, \mu) = \Big(v = \piTN(n, \xi), w = \piTN\big(n, \mu + [\xi^{[\k]}, \xi^{[\p]}]\big)\Big)
  $$
  Then the following diagram commutes:
  $$
  \xymatrix{
   J^2G \ar[r]^(.38)r\ar[dd]_{J^2\pi_{N}} & G \times \g \times \g \ar[d]^{/K}\\
   & N \times \g \times \g \ar[d]^{\piJN}\\
   J^2N \ar[r]^(.38){\sim} & TN \times_N TN
  }
  $$
 \end{lemma}
 \begin{proof}
  Starting from a $n(t) = J^2\pi_N \circ r^{-1}(g, \xi, \mu)$, and recalling that the canonical connection is given by $\Dcan = D - [\beta, \cdot]$, its image in $TN\times_N TN$ is $(v, w)$ where by definition $\beta_N(v) = \xi^{[\p]}$ and $\beta_N(w)$ is the projection to $[\p]$ of $\mu - [\beta(\dev{}{t}), \xi] = \mu - [\beta_N(v), \xi]$. In turn, this equals $\mu - [\xi^{[\p]}, \xi^{[\k]}]$, which is the same as above.
 \end{proof}

 To define equivariant deformations (which are maps form $\tM$ to $TN\times_N TN$), either one follows the diagram of Lemma \ref{lemma:diagram} to identify actions, or one works out the formulas on the right hand side of the diagram only by considering a smooth family $f_t \colon \tM \to N$ of $\rho_t$-equivariant maps and then checks that they match trough $\piJN$; either way involves some computations. We limit ourselves to giving the resulting formulas; the details can be found in \cite{thesis}, \S 5.8.
 \begin{defn}\label{defn:wequivariant}
  A second order deformation $(v, w) \in \Cinfty(f^*(TN \times_N TN))$ is called equivariant if $v$ is, as a first order deformation, and $w$ satisfies:
  \begin{align*}
   w(\gamma \tx) &= \rho_0(\gamma)_* w(\tx)\\
                 &\phantom{=}\ + \piTN \Big({f(\gamma\tx)}, k(\gamma) + 2 \big[ c(\gamma)_{\gamma\tx}^{[\k]}, \Ad_{\rho_0(\gamma)}\beta_N(v(\tx))\big] + \big[c(\gamma)_{\gamma\tx}^{[\k]}, c(\gamma)_{\gamma\tx}^{[\p]}\big]\Big).
  \end{align*}
  A function $(F, F_2) \colon \tM \to \g \times \g$ is called $\rhotsecond$-equivariant if $F$ is $\rhotfirst$-equivariant and
  $$
  F_2(\gamma \tx) = \Ad_{\rho_0(\gamma)} F_2(\tx) + \big[c(\gamma), \Ad_{\rho_0(\gamma)} F(\tx)\big] + k(\gamma).
  $$
 \end{defn}

 \begin{defn}\label{defn:wharmonic}
  A second order deformation $(v, w)$ of $f$ is harmonic if $\calJ(v) = 0$ and, in terms of a local orthonormal frame $\{E_j\}$,
  $$
  \calJ(w) = 4 \sum_j R^N\big(\de f(E_j), v\big) \nN_{E_j} v.
  $$
  Fix $\omega$ as in Notation \ref{not:omega}; then we define two operators $D_2 \colon \calA_{\tM}^p (\g \times \g) \to \calA_{\tM}^{p+1}(\g \times \g)$ and $D_{2,*} \colon \calA_{\tM}^1 (\g \times \g) \to \calA_{\tM}^0(\g \times \g)$ and functions of harmonic type $(F, F_2)$ by:
  $$
   D_2 = \begin{pmatrix}\de&0\\ \ad(\omega)& \de\end{pmatrix}, \quad D_{2,*} = \begin{pmatrix}\de^*&0\\ \omega^*\cont& \de^*\end{pmatrix}, \quad D_{2,*} D_2 \begin{pmatrix}F\\F_2\end{pmatrix} = 0,
  $$
  where $\omega^* \cont$ denotes contraction by the adjoint $\omega^* = \omegap - \omegak$ of $\omega$: In terms of a local orthonormal frame $\{E_j\}$ and for $\tilde \alpha \in \calA_{\tM}^1(\g)$,
  $$
   \omega^* \cont \talpha = \sum_j \big[ \tomega(E_j)^{[\p]} - \tomega(E_j)^{[\k]}, \talpha(E_j) \big].
  $$
 \end{defn}
 \begin{remark}
  The operator $D_2$ actually defines a flat connection on $\tM \times \g \times \g$. The contraction $\omega^* \cont$ is defined so that for every $\calV$-valued $1$-form $\alpha$ and every section $\xi$ of $\calV$, we have $\scal{[\omega, \xi]}{\alpha} = \scal{\xi}{\omega^*\cont \alpha}$.
 \end{remark}

 The proof that if $f_t$ is a family of harmonic maps then defining $(v, w)$ as in \eqref{eqn:secondorder} gives a harmonic second order deformation follows the same lines as Lemma \ref{lemma:jacobi}, covariantly differentiating the expression found for $\covt \tau(f_t)$ and using the local symmetry condition $\nN(R^N) = 0$ whenever needed.
 \begin{lemma}\label{lemma:FF2vw}
  Let $(F, F_2) \colon \tM \to \g \times \g$ be $\rhotsecond$-equivariant and of harmonic type. Then defining $(v, w) = \piJN(F, F_2)$ gives a $\rhotsecond$-equivariant and harmonic second order deformation of $f$.
 \end{lemma}
 \begin{proof}
  The proof that equivariance conditions match is tedious but straightforward (it is more agile to prove that $\beta_N(w)$ has the same kind of equivariance as $F_2^{[\p]} + [\Fk, \Fp]$; for details, we refer to \cite{thesis}, \S 5.2). To prove that $(v, w)$ is harmonic, first observe that our hypotheses force $\de F = \omega$. Then, harmonic type condition gives the following expression for $J(F_2)$:
  \begin{equation}\label{eqn:JF2}
  J(F_2) = \sum_j \big[ \omega(E_j), \nt_{E_j} F\big] - \omega^* \cont \omega
  \end{equation}
  (recall that $\tilde D = \Dcan - \ad(\beta)$). Writing $\omega = \Dcan F + [\beta, F]$, and substituting everywhere in \eqref{eqn:JF2}, one gets
  \begin{align*}
   J(F_2^{[\p]}) = J(F_2)^{[\p]} &= \sum_j 4\big[[\beta(E_j), F^{[\p]}], \nW_{E_j}F^{[\p]}\big] + 2 \big[\nW_{E_j}F^{[\k]}, \nW_{E_j}F^{[\p]}\big]\\
                 &\phantom{=\sum}\ - 2 \Big[ \big[\beta(E_j), F^{[\k]}\big], \big[\beta(E_j), F^{[\p]}\big]\Big].
  \end{align*}
  Recalling that $J(\xi) = -\trace(\tilde D \de \xi)$ and that $J(F) = 0$, we obtain:
  $$
   J\big([\Fk, \Fp]\big) = \sum_j  2 \Big[ \big[\beta(E_j), F^{[\k]}\big], \big[\beta(E_j), F^{[\p]}\big]\Big]- 2 \big[\nW_{E_j}F^{[\k]}, \nW_{E_j}F^{[\p]}\big].
  $$
  Adding the two expressions together one obtains exactly
  $$
  J(\beta_N(w)) = 4 \sum_j \big[ [\beta(E_j), \Fp], \ncan_{E_j}\Fp\big].
  $$
 \end{proof}

 \subsection{Construction of $F_2$}
 
 From now on, we are given a second order deformation $\rhotsecond$ of $\rho_0$ and a harmonic and $\rho_0$-equivariant $f \colon \tM \to N$, and we try to construct a second order deformation of $f$. By lemma \ref{lemma:FF2vw}, it is enough to construct a $(F, F_2)$ both $\rhotsecond$-equivariant and of harmonic type. Remark that the centralizer $H$ (cfr. Notation \ref{not:G0}) acts on $H^1(M, \Ad(\rho_0))$ by conjugation. This action preserves the subspace of harmonic 1-forms: Indeed, if $\omega$ is harmonic and $h \in H \cap K$ it is easy to see that $\Ad_h(\omega)$ is still harmonic (after reducing from $f$ to $f_0$, $h \in K$ becomes unitary). To conclude, since $H = H^\circ \cdot (H \cap K)$, we need only to prove that the adjoint action of $\h$ preserves harmonic forms, i.e. that for every $\xi \in \h$ the 1-form $[\omega, \xi]$ is still harmonic. This follows as in the proof of point \eqref{item:globalsections} of Proposition \ref{prop:phls}, since $\xi$ is both $\Dcan$-closed and satisfies $[\beta, \xi] = 0$.
 
 Denote by $\Ad(\rhotfirst)$ the local system given by the adjoint action of $\Gamma$ on $\g \otimes \R[t]/(t^2)$, as in \eqref{eqn:action2}. Then we have an exact sequence of sheaves:
 \begin{equation}\label{eqn:exactsequence}
  0 \to \Ad(\rho_0) \textrightarrow{\times t} \Ad(\rhotfirst) \textrightarrow{\textnormal{mod } t} \Ad(\rho_0) \to 0.
 \end{equation}
 \begin{lemma}\label{lemma:hp}
  In the long exact sequence associated to \eqref{eqn:exactsequence}, the image of the map $H^0(M, \Ad(\rhotfirst)) \to H^0(M, \Ad(\rho_0)) = \h$ is the subspace $\h'\subseteq \h$ made of those $\xi \in \h$ such that $[\omega, \xi] = 0$.
 \end{lemma}
 \begin{proof}
  The condition for $\xi + t \mu$ to be a global section of $\Ad(\rhotfirst)$ is that $\xi \in \h$ and $\Ad_{\rho_0(\gamma)}\mu = \mu - [c(\gamma), \xi]$. This last condition can be rewritten as $[c, \xi] = \delta(\mu)$, where $\delta$ denotes the coboundary in group cohomology. This means exactly $\xi \in \h'$.
 \end{proof}

 Suppose we already have $(F, F_2)$ that is both $\rhotsecond$-equivariant and of harmonic type. Then we can define a 1-form $\psi = \psi(F, F_2) \in \calA_M^1(\calV)$ by
 $$
  \begin{pmatrix}\omega\\\psi\end{pmatrix} = D_2 \begin{pmatrix}F\\F_2\end{pmatrix} = \begin{pmatrix} \de F\\ \de F_2 + [\omega, F]\end{pmatrix}.
 $$
 By flatness of $D_2$ and harmonic type condition we obtain equations for $\psi$:
 \begin{equation}\label{eqn:depsi}
  \begin{split}
   \de \psi &= - [\omega, \omega];\\
   \de^* \psi &= -\omega^* \cont \omega = 2 \cdot \sum_j \big[\omega(E_j)^{[\k]}, \omega(E_j)^{[\p]}\big] \in \Cinfty(M, [\p]).
 \end{split}
 \end{equation}
 Thus the existence of a solution to \eqref{eqn:depsi} is a necessary condition for the existence of $(F, F_2)$ as above. We shall prove that it is also sufficient (cfr. Proposition \ref{prop:obstructionF2}). First of all, we investigate on uniqueness:
 \begin{lemma}\label{lemma:uniquenessFF2}
  Let $(F, F_2)$ be $\rhotsecond$-equivariant and of harmonic type. Then every other $(F', F_2')$ both $\rhotsecond$-equivariant and of harmonic type writes as:
  \begin{equation}\label{eqn:FpF2p}
   (F', F_2') = \big(F+\xi, F_2 + [F, \xi] + \eta\big),
  \end{equation}
  where $\xi$, $\eta$ are in $\h$. Conversely, every such expression gives a $\rhotsecond$-equivariant function of harmonic type. In particular, the 1-form $\psi = \psi(F, F_2)$ is unique if and only if $\h = \h'$.
 \end{lemma}
 \begin{proof}
  One checks readily that $(F', F_2')$ defined as in \eqref{eqn:FpF2p} is both equivariant and of harmonic type (for the latter, one finds that $\psi(F', F_2') = \psi(F, F_2) + 2 [\omega, \xi]$ and uses that $[\omega, \xi]$ is harmonic). Theorem \ref{thm:firstorder} states that necessarily $F' = F + \xi$ for some $\xi \in \h$. One then reduces to $F = F'$, in which case $F_2 - F_2' \colon \tM \to \g$ becomes a $(\rho_0, k)$-equivariant map of harmonic type, and one applies again the same theorem. Then, the condition for uniqueness of $\psi$ is that for every $\xi \in \h$, $[\omega, \xi] = 0$, i.e., $\h = \h'$.
 \end{proof}

 Now we investigate the existence of a solution to \eqref{eqn:depsi}. It is well known (cfr. \cite{GoMi88}, \S 4.4) that the condition for $[\omega, \omega]$ to be null in cohomology is implied by the representation $\rhotfirst$ extending to the second order to some $\rhotsecond$. Thus, under our hypotheses, we can always find at least a solution to the first equation of \eqref{eqn:depsi}. Also remark that by the Hodge theorem on Riemannian manifolds the self-adjoint operator $J$ determines an orthogonal splitting
 $$
 \Cinfty(\calV) = \h \oplus \Image(J),
 $$
 since $\h = \ker(\de) = \ker(\de^*\de)$. Furthermore, this splitting is compatible with projections to $[\p]$ and $[\k]$, by point \eqref{item:globalsections} of Proposition \ref{prop:phls}. We can now prove the main result about the existence of $(F, F_2)$:
 
 \begin{prop}\label{prop:obstructionF2}
  The following are equivalent:
  \begin{enumerate}
   \item\label{item:soldepsi} The system of equations \eqref{eqn:depsi} admits a solution;
   \item\label{item:orthogonality} The section $\omega^*\cont \omega \in \Cinfty(\calV)$ is orthogonal to $\h$;
   \item\label{item:criticalpoint} The harmonic 1-form $\omega$ is a critical point for the $L^2$-norm in its $H$-orbit;
   \item\label{item:FF2exists} There is a pair $(F, F_2)$ which is both $\rhotsecond$-equivariant and of harmonic type;
   \item\label{item:everyFextends} Every $F \colon \tM \to \g$ both $\rhotfirst$-equivariant and of harmonic type extends to a $(F, F_2)$ as in point \eqref{item:FF2exists}.
  \end{enumerate}
 \end{prop}
 \begin{proof}
  Let $\xi$ be in $\h$. Then, if $\psi$ is a solution to \eqref{eqn:depsi}, $\scal{\omega^*\cont \omega}{\xi} = - \scal{\psi}{\de \xi} = 0$, hence \eqref{item:soldepsi}$\implies$\eqref{item:orthogonality}. Furthermore, since $\scal{\omega^* \cont \omega}{\xi} = -\frac{1}{2} \dev{}{t} \|\Ad_{\exp(t \xi)}(\omega) \|\big|_{t=0}$, we have \eqref{item:orthogonality}$\iff$\eqref{item:criticalpoint}. The implication \eqref{item:FF2exists}$\implies$\eqref{item:soldepsi} is the definition, and \eqref{item:FF2exists}$\iff$\eqref{item:everyFextends} is Lemma \ref{lemma:uniquenessFF2}. We are left with proving that \eqref{item:orthogonality} implies \eqref{item:FF2exists}. Start from an $F^0 \colon \tM \to \g$ that is both $\rhotfirst$-equivariant and of harmonic type, so that $\de F^0 = \omega$. There exists a $\g$-valued 1-form $\omega_2^0$ such that $\omega + t \omega_2^0$ is closed, $\Ad(\rhotfirst)$-valued and it represents $c + t k$ (i.e. the classes represented in $H^1(M, \Ad(\rhotfirst))$ by $\omega+t\omega_2^0$ and by $c+tk$ coincide). Define two $\calV$-valued 1-forms $\psi^0$, $\psi$ by:
  $$
  \psi^0 = \omega_2^0 - [F^0, \omega], \quad \psi = \psi^0 + \de \eta \ \textnormal{ where }\  J(\eta) = -\omega^*\cont \omega - \de^*\psi^0
  $$
  (such an $\eta$ exists because $-\omega^*\cont\omega \in \h^\perp = \Image(J)$ by hypothesis). Then $\psi$ satisfies $\eqref{eqn:depsi}$, and letting $\omega_2 = \omega_2^0 + \de \eta$, again $\omega + t \omega_2$ is an $\Ad(\rhotfirst)$-equivariant, closed 1-form that represents $c + t k$. We apply lemma \ref{lemma:existence} with this 1-form as $\phi$ and $\tau = \rhotfirst$, to construct a $F + t F_2 \colon \tM \to \g \otimes \R[t]/(t^2)$. This pair $(F, F_2)$ is then $\rhotsecond$-equivariant and of harmonic type.
 \end{proof}
 Remark that by the usual theory of moment maps (cfr. \cite{Ki84}, Part 1), when $G$ is a complex group, defining the moment map $\mu(\omega)(\xi) = -\frac{i}{2} \int_M \scal{[\xi, \omega]}{\omega}\de \Vol$ for $\xi \in \hk$, we find that point \eqref{item:criticalpoint} can be strengthened to ``$\omega$ is a minimum of the $L^2$-norm''. Such a minimum exists if and only if $\omega$ is a polystable point of the action.

 \subsection{Existence of $w$}
 We have discussed the existence of an equivariant pair $(F, F_2)$ of harmonic type, since its existence would imply the existence of a harmonic and equivariant second order deformation $w$. Now we investigate the existence of $w$ directly. For the sake of brevity, we introduce the following terminology (the reason of which will become clear in Proposition \ref{prop:complexobstructions}).

 \begin{defn}\label{defn:deformability}
  A map $f \colon \tM \to N$ is \emph{deformable} along $\rhotfirst = (\rho_0, c)$ if there exists second order deformations $\rhotsecond$ of $\rhotfirst$ and $(v, w)$ of $f$, the latter being $\rhotsecond$-equivariant and harmonic. It is called \emph{$\C$-deformable} if there exists a $\rhotsecond$ as above and a $\rhotsecond$-equivariant $(F, F_2)$ of harmonic type.
 \end{defn}
 We claim that the existence of these objects only depends on $f$ and $\rhotfirst$ only (i.e. not on the chosen $\rhotsecond$ nor on the first order deformation of $f$). That the existence of $F_2$ is independent on the $F$ chosen, is lemma \ref{lemma:uniquenessFF2}. The fact that the existence of $w$ depends on $\rhotfirst$ only has a similar proof: If $\rhotsecond$ and $\tilderhotsecond$ are two second order deformations of $\rhotfirst$, then the equations for the corresponding $w$ and $\tilde w$ are such that $w - \tilde w$ is a $(\rho_0, \tilde{k}-k)$-equivariant \emph{first order} deformation, hence we apply theorem \ref{thm:firstorder}. Now fix $\rhotsecond$, and suppose that there exists a second order deformation $(v, w)$ of $f$. Let $v'$ be any other $\rhotfirst$-equivariant harmonic first order deformation of $f$. Then there exists a $\xi \in \h$ such that $v = \piTN(f, F)$ and $v' = \piTN(f, F + \xi)$. One checks easily that
 $$
 w' = w + 2[\Fk, \xip] + [\xik, \xip]
 $$
 makes $(v', w')$ into a $\rhotsecond$-equivariant harmonic second order deformation. This concludes the proof that Definition \ref{defn:deformability} is well posed.
 
 In the following, suppose that $G$ is a complex algebraic group. Recall that in this case multiplication by $i$ anticommutes with adjunction, since $i [\k] = [\p]$. Then we have:
 \begin{prop}\label{prop:complexobstructions}
  Let $G$ be a complex group and consider the two first order deformations of $\rho_0$ given by $\rhotfirst = (\rho_0, c)$ and $\tilderhotfirst = (\rho_0, ic)$. Then $f$ is $\C$-deformable along $\rhotfirst$ if and only if it is along $\tilderhotfirst$. Furthermore, this is equivalent to $f$ being deformable \emph{both} along $\rhotfirst$ and $\tilderhotfirst$.
 \end{prop}
 \begin{proof}
  Since $(i\omega)^* = -i \omega^*$, we have $(i\omega)^*\cont(i\omega) = \omega^*\cont \omega$. Thus the condition \eqref{item:orthogonality} of Proposition \ref{prop:obstructionF2} is invariant under passing from $\rhotfirst$ to $\tilderhotfirst$. Alternatively, one can explicitly compute that if $\rhotsecond = (\rho_0, c, k)$ is a second order deformation of $\rhotfirst$, then $\tilderhotsecond = (\rho_0, ic, -k)$ is one of $\tilderhotfirst$, and if $(F, F_2)$ are what we seek for the former, then $(\tilde F, \tilde F_2) = (iF, -F_2 -\eta)$ are for the latter, for any $\eta$ such that $J(\eta) = 2 \omega^*\cont \omega$.
  
  To prove that if $f$ is deformable along both directions then it is $\C$-deformable, suppose that $(v, w)$ are defined along $\rhotsecond$ and $(\tilde v, \tilde w)$ along $\tilderhotsecond$ (defined as above). A long but straightforward computation, then, proves that defining
  $$
  -\eta = \beta_N(w) + 2i \big[\beta_N(\tilde v), \beta_N(v)\big] + \beta_N(\tilde w)
  $$
  then $J(\eta) = 2 \omega^*\cont\omega$, as claimed.
 \end{proof}

 So far, we do not know of any example of a deformable $f$ which is not $\C$-deformable. There are, however, plenty of obstructed (i.e. not deformable) first order deformations:
 \begin{example}\label{ex:obstruction}
  Consider the trivial representation $\rho_0 \colon \Gamma \to \SL(n, \R)$ and a second order deformation $\rhotsecond$ such that $c$ and $k$ are strictly upper triangular. The metrics $f \colon \tM \to N = \SL(n, \R)/O(n)$ are constant maps, hence $\beta = 0$ and the canonical connection is just flat differentiation. A short computation proves that $\de^*$ is independent of the chosen metric and $\Delta = \de^*\de$ is (up to sign) the usual Laplace-Beltrami operator on $M$. Thus harmonic first and second order deformations $\beta_N(v)$, $\beta_N(w)$ are just matrices with harmonic functions as entries. If $f$, $f' = g\cdot f$ are two metrics, with $g \in G$, ones sees easily that multiplying by $g$ sends $(\rho_0, c, k)$-equivariant and harmonic deformations of $f$ to $(\rho_0, \Ad_g(c), \Ad_g(k))$-equivariant and harmonic deformations of $f'$, so we can suppose $f \equiv eK$; in this way, $\beta_N(v)$ and $\beta_N(w)$ are symmetric matrices. Writing down explicitly the equivariance conditions, we obtain, for example for the first component $w_{11}$ of $w$:
  $$
   w_{11}(\gamma \tx) = w_{11}(\tx) + \sum_{j=2}^n\lambda_{1j}(\gamma) F_{1j}(\tx) + \frac{1}{2} \sum_{j=2}^n\lambda_{1j}(\gamma)^2; \qquad \Delta(w_{11}) = 0,
  $$
  where $\lambda_{ij}(\gamma)$ are the components of $c(\gamma)$ and $F_{ij}$ is the $(i,j)$-th component of the upper triangular matrix $F$ with $\de F = \omega$ (whose symmetrization is $\beta_N(v)$). However, one sees readily that this is the same kind of equivariance as that of $\frac{1}{2}\sum_{i=1}^n F_i(\cdot)^2$, which is subharmonic. Thus the difference $\frac{1}{2}\sum_{i=1}^n F_i(\cdot)^2 - w_{11}$ is a subharmonic function defined on $M$, which is compact, hence constant. It follows that $\frac{1}{2}\sum_{i=1}^n F_i(\cdot)^2$ is harmonic, as well, which forces it to be constant and all of the $\lambda_{1j}$'s to vanish. Proceeding inductively on the other diagonal members $w_{jj}$, one eventually finds out that, unless $c = 0$, no $\rhotsecond$-equivariant harmonic second order deformation can exist.
 \end{example}
 Clearly, the same proof works with $\SL(n, \C)$ in place of $\SL(n,\R)$, but in any of these examples if $f$ is not deformable along $(\rho_0, c)$ then it is not along $(\rho_0, ic)$ as well. Some other example has to be investigated in order to find a deformable non $\C$-deformable $f$.

 \subsection{Conclusions}
 
 To conclude, we collect the main results in the following theorem, and then we investigate the conditions on $\rhotfirst$ for which every $\rho_0$-equivariant metric is deformable to the second order.
 
 \begin{theorem}\label{thm:conclusions2}
 Let $\rhotsecond = (\rho_0, c, k)$ be a second order deformation of $\rho_0$, and $f$ a harmonic metric. If one of the equivalent conditions in proposition \ref{prop:obstructionF2} holds, then the map
\begin{align*}
\piJN \colon \bigg\{ \begin{pmatrix}F\\F_2\end{pmatrix} \begin{array}{l} \rhotsecond \text{-equivariant }\\\text{of harmonic type}\end{array}\bigg\} &\longrightarrow
\bigg\{\begin{pmatrix}v\\w\end{pmatrix} \begin{array}{l} \rhotsecond \text{-equivariant}\\ \text{and harmonic}\end{array}\bigg\}\\
   \begin{pmatrix}F\\F_2\end{pmatrix} &\longmapsto
\Big(\piTN(f, F), \piTN\big(f, F_2 + [F^{[\k]}, F^{[\p]}]\big)\Big)
\end{align*}
 is surjective, and in fact every $\rhotfirst$-equivariant and harmonic first order deformation $(f, v)$ extends to a second order $\rhotsecond$-equivariant and harmonic $(f, v, w)$. When $G$ is a complex algebraic group, the condition above is equivalent to the existence of two harmonic and equivariant second order deformations, one along $(\rho_0, c)$ and the other along $(\rho_0, ic)$. In this case, up to changing $f$ to $h^{-1}f$, for some $h \in H$, the condition can be satisfied if and only if the orbit $H \cdot \omega$ is closed in $\calH^1(M, \Ad(\rho_0))$.
 \end{theorem}

 \begin{prop}
  Let $G$ be a complex algebraic group and $\rhotfirst$ a first order deformation of $\rho_0$. Then the following conditions are equivalent to every $\rho_0$-equivariant $f \colon \tM \to N$ being $\C$-deformable along $\rhotfirst$:
  \begin{enumerate}
   \item\label{item:flatness} The $\R[t]/(t^2)$-module $H^0\big(M, \Ad(\rhotfirst)\big)$ is flat;
   \item\label{item:exactsequence} There is an exact sequence in cohomology:
   $$
   0 \to H^0(M, \Ad(\rho_0)) \textrightarrow{\times t} H^0(M, \Ad(\rhotfirst)) \textrightarrow{\textnormal{mod } t} H^0(M, \Ad(\rho_0)) \to 0,
   $$
   i.e. $\h = \h'$ (otherwise said, the orbit $H \cdot \omega \subseteq \calH^1(M, \Ad(\rho_0))$ is discrete).
  \end{enumerate}
 \end{prop}
 \begin{proof}
  We start by proving that \eqref{item:flatness}$\iff$\eqref{item:exactsequence}. The only non trivial ideal of $A = \R[t]/(t^2)$ is $(t)$, so writing $\calM = H^0(M, \Ad(\rhotfirst))$, flatness is equivalent to the injectivity of $(t)\otimes_A \calM \to \calM$, that is, to the proposition:
  $$
  \forall t\eta \in \calM, \quad t \otimes t \eta = 0 \in (t) \otimes_A \calM.
  $$
  Now $t \otimes t\eta = 0$ in $(t) \otimes_A \calM$ if and only if $\eta \in \calM$, while $t \eta \in \calM$ is equivalent to $\eta \in \h$. Thus flatness is equivalent to $\h \subseteq \calM$, as wanted. Then, if we assume point \eqref{item:exactsequence}, that is, $\h = \h'$, we have $[\omega, \xi] = 0$ for all $\xi \in \h$, hence $\scal{\omega}{[\omega, \xi]} = 0$. Thus $\omega^* \cont \omega \perp \h$, and Proposition \ref{prop:obstructionF2} implies $\C$-deformability.

  To prove the converse (namely that flatness is also necessary for the $\C$-deformability of every metric), we make use of the theory of moment maps. By hypothesis, $\omega$ must be a critical point for every metric; since the metrics are of the form $h \cdot f_0$, for a fixed $f_0$ and $h \in H$, this is tantamount to saying that $\Ad_h(\omega)$ must be critical in the norm induced by $f_0$ for every $h$, that is, that the orbit has constant $L^2$-norm. But the minimal locus is also a $(H \cap K)$-orbit, hence we get $H = (H \cap K) \cdot H'$, where $H'$ is the subgroup of $H$ fixing $\omega$ (its Lie algebra is thus $\h'$); at the level of Lie algebras, $\h = \hk + \h'$. Furthermore, we know that $\h' = (\h' \cap \hk) \oplus (\h' \cap \hp)$ is a reductive and complex Lie algebra. By elementary linear algebra, these facts together imply that $\h = \h'$.
 \end{proof}

 \section{The second variation of the energy functional}\label{sec:secondvariation}
 
 Let $\rhotsecond$ be a second order variation of $\rho_0$, and suppose that $(v, w)$ is a $\rhotsecond$-equivariant second order deformation of $f$. Remark that, given a local orthonormal frame $\{E_j\}$, the expression
 $$
 \bigscal{\nN w}{\de f} + \sum_j \Bigscal{R^N\Big(\de f(E_j), v\Big) v}{\de f(E_j)} + \big\|\nN v\big\|^2
 $$
 is $\Gamma$-invariant. The proof follows the same lines as for \eqref{eqn:firstderivative} (for details, see \cite{thesis}, Lemma 6.1.2).
 \begin{defn}
  The energy of a $\rhotsecond$-equivariant second order deformation $(f, v, w)$ is defined as:
  \begin{align}
    E(f,v,w) &= E(f) + t \int \scal{\nabla v}{\de f} + \nonumber\\
	     &\phantom{=}\ \frac{t^2}{2} \int \scal{\nN w}{\de f} + \sum_j \Bigscal{R^N\Big(\de f(E_j), v\Big) v}{\de f(E_j)} +  \big\|\nN v\big\|^2. \label{eqn:secondorderenergy}
  \end{align}
 \end{defn}
 Again, one sees directly that when $(v, w)$ are induced by a smooth family $f_t$, this coincides up to the second order with $E(f_t)$. Applying $\beta_N$ to every term in \eqref{eqn:secondorderenergy}, we find the alternative expression for the second order:
 \begin{align}
  \frac{\partial^2 E(f,v,w)}{\partial t^2}\Big|_{t=0} = \int \bigscal{\nW \beta_N(w)}{\tbeta} + \big\|[\tbeta,\beta_N(v)]\big\|^2 + \big\| \nW F^{[\p]} \big\|^2\label{eqn:secondorderenergybeta} 
 \end{align}
 \begin{prop}\label{prop:secondvariationenergy}
  Let $\rhotsecond$ be a second order deformation of $\rho_0$, and suppose that $(v, w)$ is a second order deformation of $f$ induced by a $\rhotsecond$-equivariant $(F, F_2)$ of harmonic type as in Theorem \ref{thm:conclusions2}. Set $\psi = \psi(F, F_2)$; then, the second order of the energy $E(f, v, w)$ may be written as:
  \begin{equation}\label{eqn:secondorderenergypsi}
  \frac{\partial^2 E_t}{\partial t^2}\Big|_{t=0} = \int_M \bigscal{\psi}{\beta} + \|\omega^{[\p]}\|^2.
  \end{equation}
 \end{prop}
 \begin{proof}
  Recall that, using Notation \ref{not:projection}, the relation between $(v, w)$ and $(F, F_2)$ is:
  \begin{align}
   \beta_N(v) = F^{[\p]}; \quad \beta_N(w) = F_2^{[\p]} + \big[F^{[\k]}, F^{[\p]}\big]. \label{eqn:FF2vw}
  \end{align}
  Since $\omegap = \Dcan \Fp + [\tbeta, \Fk]$, we have
  $$
  \big\|\omegap\big\|^2 = \big\|\nW \Fp\big\|^2 + \big\|[\tbeta, \Fk]\big\|^2 + 2 \bigscal{\nW \Fp}{[\tbeta,\Fk]}.
  $$
  Thus, comparing \eqref{eqn:secondorderenergybeta} to \eqref{eqn:secondorderenergypsi}, we are reduced to proving that
  \begin{equation}\label{eqn:secondvariationtmp}
   \scal{\nW \beta_N(w)}{\tbeta} + \|[\tbeta, \beta_N(v)]\|^2 = \scal{\psi}{\tbeta} + \| [\tbeta, F^{[\k]}]\|^2 + 2 \scal{\nW F^{[\p]}}{[\tbeta, F^{[\k]}]}.
  \end{equation}
  Using $\Dcan F_2 + [\tbeta, F_2] = D F_2 = \psi + [F, \omega]$ and \eqref{eqn:FF2vw}:
  $$
   \nW \beta_N(w) = \Big( \psi + [F, \omega] - [\tbeta, F_2] \Big)^{[\p]} + [\nW F^{[\k]}, F^{[\p]}] + [F^{[\k]}, \nW F^{[\p]}].
  $$
  Substituting this expression into \eqref{eqn:secondvariationtmp}, and writing $\omega = \Dcan F + [\tbeta, F]$ in terms of $\Fk$ and $\Fp$ gives the result (remark also that $\scal{[\tbeta, F_2]}{\tbeta} = 0$ as in the proof of Proposition \ref{prop:firstvariation}).
 \end{proof}
  When $M = (X, \Omega)$ is a K\"ahler manifold the expression in \eqref{eqn:secondorderenergypsi} is independent of the metric chosen in its K\"ahler class, as it follows from the next lemma (which also gives a different proof for the analogous statement for the first order).
 \begin{lemma}\label{lemma:independence}
  Let $\alpha_1$, $\alpha_2$ be two $\g$-valued 1-forms on a compact K\"ahler manifold $(X, \Omega)$, and suppose that at least one of them takes values in the subbundle $[\p]$. Then their $L^2$ product
  \begin{equation}\label{eqn:L2product}
   \int_X \bigscal{\alpha_1}{\alpha_2}\ \Omega^n = \int_X \big\langle \alpha_1 \wedge *\alpha_2 \big \rangle
  \end{equation}
  is independent of the metric chosen in the K\"ahler class $\Omega$.
 \end{lemma}
 \begin{proof}
  Without loss of generality, both $\alpha_1$ and $\alpha_2$ take values in $[\p]$. Hence $\alpha_2 = \varphi + \varphi^*$, where $\varphi$ is the $(1, 0)$-part of $\alpha_2$. We get $*\alpha_2 = (\varphi^* - \varphi) \wedge \Omega^{n-1}$. Thus, up to some constant, \eqref{eqn:L2product} is $\int_X \trace(\alpha_1 \wedge (\varphi - \varphi^*)) \wedge \Omega^{n-1}$. We are only left to prove that $\trace(\alpha_1 \wedge (\varphi - \varphi^*))$ is a closed 2-form. Now $\varphi - \varphi^*$ takes values in $[\k^\C] = [\k \oplus i\p]$, the anti-selfadjoint part of $\g \otimes \C$. Then, by orthogonality:
  $$
  \trace \Big( \de \big (\alpha_1 \wedge (\varphi - \varphi^*)\big) \Big) = \trace \Big(\big[\beta, \alpha_1] \wedge (\varphi - \varphi^*)\Big) - \trace \Big( \alpha_1 \wedge \big[ \beta, \varphi - \varphi^*\big]\Big).
  $$
  Combining the cyclic symmetry of the trace with the basic symmetry for every two 1-forms $[\alpha', \alpha''] = [\alpha'', \alpha']$, this expression vanishes.
 \end{proof}
 \begin{example}
  Continuing the case $G = \C^*$ from Example \ref{ex:C*}, one sees that in that case $\psi^{\p}$ is simply $\frac{\partial^2 \tbeta_t}{\partial t^2}\big|_{t=0}$, thus equation \eqref{eqn:secondorderenergypsi} becomes the following trivial expression:
  $$
   \frac{\partial^2E(t)}{\partial t^2}\Big|_{t=0} = \int_X \dev{}{t}\Bigscal{\dev{\tbeta}{t}}{\tbeta_t}\Big|_{t=0}\de\Vol = \int_X \bigg(\Bigscal{\frac{\partial^2\tbeta_t}{\partial t^2}\Big|_{t=0}}{\tbeta} + \bigg\|\dev{\tbeta_t}{t}\Big|_{t=0}\bigg\|^2 \bigg)\de \Vol.
  $$
 \end{example}
 
 \subsection{Plurisubharmonicity of the energy}
 
 \begin{defn}
  When $G$ is a complex group, the Betti complex structure on the tangent space $Z^1(\Gamma, \g)$ to the representation space $\scrR(\Gamma, G)$ at $\rho_0$ is defined, for every 1-cocycle $c$, by $J_B(c) = ic$. The metric on the same tangent space is given by the $L^2$ scalar product of the corresponding harmonic representatives in $\calH^1(M, \Ad(\rho_0))$.
 \end{defn}

 \begin{theorem}\label{thm:kahlerpotential}
  Let $G$ be a complex group. The energy functional is a K\"ahler potential for the K\"ahler structure on the moduli space $\MB(M, G)$. In particular, the energy functional is plurisubharmonic on the smooth points of $\MB(M, G)$, and thus defines a plurisubharmonic function on the normalization of the moduli space.
 \end{theorem}
 \begin{proof}
  For the consequence on the normalization, see the argument in \cite{FoNa80}, section 3. Let $\dev{}{x} \in T_{\rho_0}\Hom(\Gamma, G)$ be a tangent direction to the representation space, and define as usual $\dev{}{y} = J_B \dev{}{x}$. Then we need to prove that
  $$
  \de \de^c E \Big(\dev{}{x}, \dev{}{y}\Big) = \Big(\devd{}{x}+ \devd{}{y}\Big) E = \Big\| \dev{}{x} \Big \|^2.
  $$
  In the same notation we have used so far, this reduces our proof to showing that, for every harmonic map $f$ deformable both along $\dev{}{t}$ and along $J_B\dev{}{t}$, we have 
  $$
  \bigg(\frac{\partial^2}{\partial t^2} + \Big(J_B \dev{}{t}\Big)^2 \bigg) E = \int_M \|\omega\|^2 \de \Vol_g \geq 0.
  $$

  By Theorem \ref{thm:conclusions2}, $f$ is $\C$-deformable, hence $(v, w) = \piJN(F, F_2)$ for some $(F, F_2)$. As in the proof of Proposition \ref{prop:complexobstructions}, if $\eta$ is such that $J(\eta) = 2 \omega^*\cont \omega$, then the pair $(F', F_2') = (iF, -F_2 -\eta)$ induces the deformation along $(\rho_0, ic)$. The corresponding 1-form is $\psi(F', F_2')= -\psi -\de \eta$. Thus we compute:
  \begin{align*}
   \bigg(\frac{\partial^2}{\partial t^2} + \Big(J_B \dev{}{t}\Big)^2 \bigg) E(f_t) &= \int \scal{\psi}{\beta} + \big\|\omegap \big\|^2 + \scal{-\psi-\de \eta}{\beta} + \big\|(i\omega)^{[\p]}\big\|^2\\
             &= \int \big\|\omegap \big\|^2  - \scal{\de \eta}{\beta} + \big\|i\omegak \big\|^2 = \int \|\omega\|^2,
  \end{align*}
  since $\de^*\beta = 0$ is one way to express the harmonicity of $f$.
 \end{proof}
 \begin{remark}
  It is a well-known consequence of Uhlenbeck's compactness theorem that the energy functional is proper on $\MB(M, G)$ (see, for example, \cite{DaDoWe98}, Proposition 2.1); this fact, combined with Theorem \ref{thm:kahlerpotential}, gives another proof that $\MB(M, G)$ is Stein.
 \end{remark}

 \subsection{Positivity of the Hessian of the energy}
 Recall that Hitchin \cite{Hi87} constructed the moduli space of solutions to the self-duality equations on a Riemann surface $\Sigma$ as the quotient of the infinite dimensional affine space $\calA \times \Omega^{0,1}$, where $\calA$ is the space of flat connections on a principal bundle $P$ (modeled on $\calA^{1,0}(\Sigma, \ad(P) \otimes \C)$) and $\Omega^{0,1} = \calA^{0,1}(\Sigma, \ad(P) \otimes \C)$. Tangent vectors to the moduli space lift to pairs $(\dot{A}, \dot{\Phi})$ belonging to the associated vector space. One can see easily that in our notations a direction determined by $\omega$ corresponds to
 \begin{equation}\label{eqn:harmonictoHiggs}
 \dot{A} = \big(\omegak\big)''; \quad \dot{\Phi} = \big(\omegap\big)',
 \end{equation}
 where $\alpha = \alpha' + \alpha''$ stands for the $(1,0)$ and $(0,1)$ parts of a 1-form $\alpha$ and $\alphak$ is such that $\alphak(\chi) = \alpha(\chi)^{[\k]}$ for real tangent fields $\chi \in \Xi(\Sigma)$ (for the case of $G = \C^*$, cfr. \cite{GoXi08} \S 4.2 for this passage ``from harmonic coordinates to Higgs coordinates''). For general K\"ahler manifolds $X$, we will take \eqref{eqn:harmonictoHiggs} as a definition of $\dot{A}$, $\dot{\Phi}$ and we aim to generalize the result in \cite{Hi92}, \S 9. What we prove is the following:
 \begin{theorem}\label{thm:hessian}
  Let $G$ be a complex algebraic group, and suppose that $\rho_0 \colon \Gamma \to G$ is representation which is induced by a $\C$-VHS (i.e. a critical point of the energy). Denote by $f_0 \colon \tM \to G_0/K_0 \subset G/K$ the map induced by the period mapping, as in Notation \ref{not:G0}. Then, denoting by $(\dot{A}, \dot{\Phi})$ a tangent direction to the moduli space and by $\g = \bigoplus [\g^{-p,p}]$ the Hodge structure on $\tM \times \g$, along $\C$-deformable directions we have
  \begin{equation}\label{eqn:secondvariationHodgeHitchin}
   \devd{E(f_t)}{t}\Big|_{t=0} = 2\int_X \sum_p \Big(-p \big\| \dot{A}^{-p,p} \big\|^2 + (1-p) \big\| \dot{\Phi}^{-p,p} \big\|^2\Big)\de \Vol.
  \end{equation}
 \end{theorem}
 \begin{cor}\label{cor:variationG0}
  If we suppose further that the deformation takes place in $G_0$ only, that is, that $\omega \in \calA^1(\g_0)$, then the following more convenient expressions are available (the last two are in terms of the weight 1 Hodge-Deligne $(P,Q)$-decomposition of $\calH^1(M, \Ad(\rho_0))$, cfr. \cite{Zu79}):
  \begin{align*}
   \devd{E(f_t)}{t}\Big|_{t=0} &= 2 \int_X \sum_p c_p \big\| (\omega')^{-p,p}\big\|^2, \quad c_p = \begin{cases} p,& \mbox{if } p \mbox{ is even},\\ 1-p, & \mbox{if } p \mbox{ is odd}. \end{cases} \\
   &= \int_X \sum_{P+Q = 1} c_P \|\omega^{(P,Q)}\|^2 = \int_X \sum_{P \textnormal{ even}} 2P \| \omega^{(P,Q)} \|^2.
  \end{align*}
 \end{cor}
 \begin{proof}
  Equalities between all the stated expressions follow from the hypothesis of $\omega$ being real (i.e. in $\g_0$) by making use of
  \begin{equation*}
   \|\omega^{-p,p}\|^2 = \big\|(\omega')^{-p,p}\big\|^2 + \big\|(\omega'')^{-p,p}\big\|^2 = \big\|(\omega')^{-p,p}\big\|^2 + \big\|(\omega')^{p,-p}\big\|^2
  \end{equation*}
  and $\|(\omega'')^{p,-p}\|^2 = \|(\omega')^{-p,p}\|^2$.
 \end{proof}
 \begin{cor}
  In the moduli space $\MB(X, \PSL(2, \R))$, at every critical point the Hessian of the energy is semipositive definite.
 \end{cor}
 \begin{proof}
  At every such point, either the energy is zero or we have $\gamma \in [\g^{0,0}]$ and $\theta \in \calA^{1,0}([\g^{-1,1}])$, and since $\g$ has complex dimension $3$, there can be nothing in $[\g^{-p,p}]$ for $|p| \geq 2$. Thus any expression in Corollary \ref{cor:variationG0} proves the claim.
 \end{proof}
 \begin{defn}
  Suppose $f$ to be induced by a $\C$-VHS as above, and let $v$ be a $\rhotfirst$-equivariant and harmonic first order deformation of $f$. We say that $v$ is $\C$-VHS to the first order if $\partial \beta_N(v) = \covt \theta_t\big|_{t=0} \in \calA^{1,0}([\g^{-1,1}])$.
 \end{defn}
 One can see easily that this is the first order condition for $\beta_t = (\de f_t \cdot f_t^{-1})^{[\p]}$ to remain in $\calA^{1,0}([g^{-1,1}]) \oplus \calA^{0,1}([\g^{1,-1}])$.
 \begin{cor}
  Let $\rho_0$ be induced by a $\C$-VHS, and denote by $G_0$ its real Zariski closure. If $\rho_0$ is of Hermitian type, then the Hessian is semipositive definite along $\MB(X, G_0)$. The directions along which it vanishes are exactly those which are $\C$-VHS to the first order.
 \end{cor}
 \begin{proof}
  By Corollary \ref{cor:variationG0}, since the Hodge structure on $\g$ has only weights $\pm1$ and $0$, the second variation is $4\int_X \|(\omega')^{1,-1}\|^2 \de \Vol$. Now $\omega(\partial_j) = \partial_j F + [\theta(\partial_j), F]$, the second summand of which must take values in $[\g^{-1,1}]\oplus[\g^{0,0}]$, hence it plays no role in $\|(\omega')^{1,-1}\|^2$. Thus
  $$
  \omega(\partial_j)^{1,-1} = \partial F^{1,-1} = 0 \iff \covt \theta_t \Big|_{t=0} = \partial \Fp \in \calA^{1,0}\big([\g^{-1,1}]\big).
  $$
 \end{proof}
 
 \begin{lemma}\label{lemma:secondvariationenergyVHS}
  Let $\rhotsecond$ be a second order deformation of a representation $\rho_0$, supposed to be induced by a $\C$-VHS. Let $f$ be induced by the period mapping, as above, and suppose that it is $\C$-deformable along $\rhotfirst$. Then, the second variation of the energy reads
  \begin{equation}\label{eqn:secondvariationenergyVHS}
   \frac{\partial^2 E(f_t)}{\partial t^2} \Big|_{t=0} = \int_X \Big(\bigscal{\Lambda [\omega,\omega]}{\gamma} + \|\omegap\|^2\Big)\de \Vol.
  \end{equation}
 \end{lemma}
 \begin{proof}
  Thanks to the K\"ahler identities (cfr. \cite{Zu79} or \cite{Si92}), equation \eqref{eqn:betadec} and the identity $\de \psi = -[\omega, \omega]$ in \eqref{eqn:depsi}, we have
  \begin{align*}
  \int\scal{\psi}{\beta} \overset{\eqref{eqn:betadec}}{=} \int \scal{\psi}{D^c \gamma} = \int \scal{{D^c}^* \psi}{\gamma} \overset{\textnormal{KI}}{=} \int \scal{-[\Lambda, \de] \psi}{\gamma} \overset{\eqref{eqn:depsi}}{=} \int \scal{\Lambda[\omega,\omega]}{\gamma}.
  \end{align*}
 \end{proof}

 \begin{proof}[Proof of Theorem \ref{thm:hessian}]
  Denote as usual by $G_0$ the monodromy group, and denote by $\overline{G_0^{\C}}$ its complex Zariski closure. By hypothesis, $\rho_0$ being induced by a $\C$-VHS means that there is a faithful linear representation $\overline{G_0^{\C}} \hookrightarrow \GL(r, \C)$ such that the resulting vector bundle $\calV = (\tX \times \C^r)/\Gamma$ supports a $\C$-VHS; we give $\End(\calV) = (\tX \times \gl_n(\C))/\Gamma$ the induced $\C$-VHS structure of weight 0. Then we know that $G_0$ is the intersection of $\overline{G_0^{\C}}$ with the subgroup $U(p,q)$ of $\GL(r, \C)$ respecting the polarization (this is essentially the content of a theorem by Karpelevich and Mostow, \cite{Ka53,Mo55}), and that if we set $\ku = \u \cap \bigoplus_{p \equiv 0} \g^{-p,p}$ and $\pu = \u \cap \bigoplus_{p \equiv 1} \g^{-p,p}$ we obtain a Cartan decomposition for $\u$. We define $\k = \ku \oplus i\pu$ and $\p = \pu \oplus i\ku$ for the induced Cartan decomposition of $\g$. Then since $f$ takes values in $G_0/K_0$, every two out of the four terms of the decomposition $\g = \ku \oplus \pu \oplus i\ku \oplus i\pu$ are orthogonal with respect to the metric on $\tX \times \g$. This is twice the real part of the Hermitian extension $\scal{\cdot}{\cdot}_\C$ of the metric $\scal{\cdot}{\cdot}$ induced on $\u$ by $f$. Taking an adequate faithful representation, then, we can suppose without loss of generality that $\g = \gl_r(\C)$.
  
  Fix a local orthonormal frame on $X$ of the form $\{\dev{}{x_j}, \dev{}{y_j} = i\dev{}{x_j}\}$, and write for brevity (dropping the $j$ in the notation):
  $$
   \omega\Big(\dev{}{x_j}\Big) = \xi_1 + i \xi_2; \quad \omega\Big(\dev{}{y_j}\Big) = \eta_1 + i \eta_2, \quad \xi_1, \xi_2, \eta_1, \eta_2 \in \u.
  $$
  Write $\xi_k = \sum_p \xi_k^p$ for the projection $\xi_k^p \in [\g^{-p,p}]$ (which is no more in $\u$), and similarly for the $\eta$'s. 
  We aim to prove that both \eqref{eqn:secondvariationHodgeHitchin} and \eqref{eqn:secondvariationenergyVHS} reduce to:
  \begin{equation}\label{eqn:intermediate}
   \begin{split}
    \int_X \sum_p \bigg((-1)^p 4p \Im\Big(\bigscalC{\xidp}{\etadp} - \bigscalC{\xiup}{\etaup}\Big) &+ \sum_{p\equiv 1} 2 \normC{\xiup} + 2 \normC{\etaup}\\
                 &+ \sum_{p \equiv 0} 2 \normC{\xidp} + 2 \normC{\etadp} \bigg) \de \Vol.
   \end{split}
  \end{equation}
  Indeed, the first term of \eqref{eqn:secondvariationenergyVHS} equals the term in \eqref{eqn:intermediate} involving the imaginary part, and the second one the part involving the squares. The latter claim is proved explicitly by computing $\|\omegap\|^2 = \sum \|\xi_1^{[\p_0]}\|^2 + \|(i\xi_2)^{[i\k_0]}\|^2 + \|\eta_1^{[\p_0]}\|^2 + \|(i\eta_2)^{[i\k_0]}\|^2$, and using that $[\pu \oplus i\pu] = \bigoplus_{p\equiv 1} [\g^{-p,p}]$ to get the result. The former is a bit longer; first of all, $\int \scal{\Lambda [\omega, \omega]}{\gamma}$ equals:
  \begin{align}
  &\int_X \Re \Big( -2i \BigscalC{\xi_1 + \eta_2 + i \big(\xi_2 - \eta_1\big)}{\big[\gamma, \big(\xi_1 - \eta_2 + i \big(\xi_2 + \eta_1\big)\big)^*\big]}\Big)\nonumber\\
  &= 2 \int_X \bigscalC{\xi_1 + \eta_2}{\big[\gamma, \xi_2^* + \eta_1^*\big]} + \bigscalC{\xi_2 - \eta_1}{\big[\gamma, \xi_1^* - \eta_2^*\big]} \label{eqn:intermediate2}
  \end{align}
  (here we have disregarded the purely imaginary terms and the last expression is, in fact, real). Remark that, $\xi_1$ being real (that is, in $\u$), $\xi_1^* = \sum_p (-1)^{p+1} \xi_1^p$, and similarly for $\xi_2^*$, $\eta_1^*$ and $\eta_2^*$; thus:
  \begin{align*}
   \eqref{eqn:intermediate2} = 2 \int_X &\sum_p (-1)^p ip \bigg( \bigscalC{\xiup}{\xidp} + \bigscalC{\etadp}{\xidp} + \bigscalC{\xiup}{\etaup} + \bigscalC{\etadp}{\etaup}\bigg)\\
                     & + \sum_p (-1)^p ip \bigg( \bigscalC{\xidp}{\xiup} - \bigscalC{\xidp}{\etadp} - \bigscalC{\etaup}{\xiup} + \bigscalC{\etaup}{\etadp}\bigg).
  \end{align*}
  Since the result must be real, the terms $\bigscalC{\xiup}{\xidp} + \bigscalC{\xidp}{\xiup}$ (and the respective ones for the $\eta$'s) cancel out. This finishes the first half of the proof.
  
  To prove that \eqref{eqn:secondvariationHodgeHitchin} equals \eqref{eqn:intermediate}, the usual relations between Cartan and Hodge decompositions give
  \begin{align*}
   \big\|\dot{A}^{-p,p}\big\|^2 &= \bignormC{\omegak(2 \bar\partial_j)^{-p,p}} = \sum_{p \equiv 0} \bignormC{\xiup + i \etaup} + \sum_{p \equiv 1} \bignormC{i \xidp-\etadp};\\
   \big\|\dot{\Phi}^{-p,p}\big\|^2 &= \bignormC{\omegap(2\partial_j)^{-p,p}} = \sum_{p \equiv 0} \bignormC{i\xidp + \etadp} + \sum_{p \equiv 1} \bignormC{\xiup - i \etaup}.
  \end{align*}
  Substituting into \eqref{eqn:secondvariationHodgeHitchin} and using the identity $\normC{a + ib} = \normC{a} + \normC{b} + 2 \Im\scal{a}{b}_{\C}$, we get:
  \begin{align*}
   &2\sum_{p \equiv 0} -p \Big( \bignormC{\xiup} + \bignormC{\etaup} + 2 \Im\bigscalC{\xiup}{\etaup} \Big)
   + (1-p) \Big(\bignormC{\xidp} + \bignormC{\etadp} + 2 \Im\bigscalC{\etadp}{\xidp} \Big)\\
   &+ 2 \sum_{p \equiv 1} -p \Big( \bignormC{\xidp} + \bignormC{\etadp} + 2 \Im\bigscalC{\xidp}{\etadp} \Big)
    + (1-p) \Big(\bignormC{\xiup} + \bignormC{\etaup} + 2 \Im\bigscalC{\etaup}{\xiup} \Big).
  \end{align*}
  Finally, as $\xi_1$ is real, $\normC{\xiup} = \normC{\xi_1^{-p}}$, so that summing over positive and negative $p$'s cancel out (and similarly for the other norms). For the same reason, the terms of the form  $\Im\scal{\xidp}{\etadp}_{\C}$ cancel out \emph{unless} they are multiplied by $p$. Removing the vanishing terms gives \eqref{eqn:intermediate}.
 \end{proof}
 
 \begin{remark}
  Avoiding technical complications, the same ideas in this proof can be used directly to prove Corollary \ref{cor:variationG0} for variations inside $\g_0$ only; the resulting computations simplify significantly (cfr. \cite{thesis}, Proposition 6.3.4).
 \end{remark}

\bibliographystyle{alpha}
\bibliography{refs}
\end{document}

%% file: packages.tex
\usepackage{amssymb} \usepackage{amsfonts} \usepackage{amsmath}
\usepackage{amsthm} \usepackage{epsfig} 
\usepackage{amscd} \usepackage[all]{xy}
\usepackage{color}
\usepackage{enumerate}
\usepackage{url}
\usepackage[foot]{amsaddr}

\usepackage{stmaryrd}
\usepackage{mathabx}
\usepackage{mathtools}
\usepackage{mathrsfs}

\usepackage{graphicx}
\usepackage{caption}
\usepackage{subcaption}

\usepackage[utf8x]{inputenc}

\makeatletter
\newcommand*{\@old@slash}{}\let\@old@slash\slash
\def\slash{\relax\ifmmode\delimiter"502F30E\mathopen{}\else\@old@slash\fi}

\def\bign#1{\mathclose{\hbox{$\left#1\vbox to8.5\p@{}\right.\n@space$}}\mathopen{}}
\def\Bign#1{\mathclose{\hbox{$\left#1\vbox to11.5\p@{}\right.\n@space$}}\mathopen{}}
\def\biggn#1{\mathclose{\hbox{$\left#1\vbox to14.5\p@{}\right.\n@space$}}\mathopen{}}
\def\Biggn#1{\mathclose{\hbox{$\left#1\vbox to17.5\p@{}\right.\n@space$}}\mathopen{}}

\makeatother

\def\acts{\mathrel{\reflectbox{$\righttoleftarrow$}}}

\DeclareMathOperator*{\nablaop}{\nabla}

\newcommand{\nW} {\ensuremath{\overset{\textnormal{can}}{\nablaop}}}
\newcommand{\ncan} {\ensuremath{\overset{\textnormal{can}}{\nablaop}}}

\newcommand{\nN} {\ensuremath{\overset{N}{\nablaop}}}

\newcommand{\C} {\ensuremath{\mathbb{C}}}

\newcommand{\R} {\ensuremath{\mathbb{R}}}

\newcommand{\M} {\ensuremath{\mathbb{M}}}
\newcommand{\V} {\ensuremath{\mathbb{V}}}
\newcommand{\G} {\ensuremath{\mathbb{G}}}

\newcommand{\MB} {\ensuremath{\mathbb{M}_{\textnormal{B}}}}

\newcommand{\MDol} {\ensuremath{\mathbb{M}_{\textnormal{Dol}}}}
\newcommand{\RB} {\ensuremath{\mathbb{R}_{\textnormal{B}}}}

\newcommand{\F} {\ensuremath{\mathbf{F}}}

\newcommand{\scrH} {\ensuremath{\mathscr{H}}}
\newcommand{\scrR} {\ensuremath{\mathscr{R}}}

\newcommand{\h} {\ensuremath{\mathfrak{h}}}
\newcommand{\g} {\ensuremath{\mathfrak{g}}}
\renewcommand{\k} {\ensuremath{\mathfrak{k}}}
\newcommand{\p} {\ensuremath{\mathfrak{p}}}

\renewcommand{\v} {\ensuremath{\mathfrak{v}}}

\newcommand{\w} {\ensuremath{\mathfrak{w}}}

\newcommand{\hk} {\ensuremath{\mathfrak{h^k}}}
\newcommand{\hp} {\ensuremath{\mathfrak{h^p}}}
\newcommand{\ku} {\ensuremath{\mathfrak{k_u}}}
\newcommand{\pu} {\ensuremath{\mathfrak{p_u}}}
\renewcommand{\a} {\ensuremath{\mathfrak{a}}}
\renewcommand{\u} {\ensuremath{\mathfrak{u}}}

\newcommand{\gl} {\ensuremath{\mathfrak{gl}}}

\newcommand{\calE} {\ensuremath{\mathcal{E}}}

\newcommand{\calJ} {\ensuremath{\mathcal{J}}}

\newcommand{\calM} {\ensuremath{\mathcal{M}}}

\newcommand{\calA} {\ensuremath{\mathcal{A}}}

\newcommand{\calC} {\ensuremath{\mathcal{C}}}
\newcommand{\calW} {\ensuremath{\mathcal{W}}}
\newcommand{\calH} {\ensuremath{\mathcal{H}}}
\newcommand{\calV} {\ensuremath{\mathcal{V}}}

\newcommand{\into} {\ensuremath{\hookrightarrow}}

\newcommand{\de} {\ensuremath{\textnormal{d}}}
\newcommand{\dist} {\ensuremath{\textnormal{dist}}}

\newcommand{\Dcan} {\ensuremath{{D^{\textnormal{can}}}}}
\newcommand{\Dpb} {\ensuremath{{D^{\textnormal{pb}}}}}
\newcommand{\Dalpha} {\ensuremath{{D^{\alpha}}}}
\newcommand{\dcan} {\ensuremath{{\textnormal{d}^{\textnormal{can}}}}}

\newcommand{\Ad} {\ensuremath{\textnormal{Ad}}}
\newcommand{\ad} {\ensuremath{\textnormal{ad}}}

\newcommand{\Image} {\ensuremath{\textnormal{Image}}}

\newcommand{\SL} {\ensuremath{\textnormal{SL}}}

\newcommand{\GL} {\ensuremath{\textnormal{GL}}}

\newcommand{\trace} {\ensuremath{\textnormal{trace}}}

\newcommand{\End} {\ensuremath{\textnormal{End}}}
\newcommand{\Hom} {\ensuremath{\textnormal{Hom}}}

\newcommand{\Pic} {\ensuremath{\textnormal{Pic}}}

\newcommand{\can} {\ensuremath{\textnormal{can}}}
\newcommand{\std} {\ensuremath{\textnormal{std}}}
\newcommand{\PSL} {\ensuremath{\mathbb{P}\SL}}

\renewcommand{\Re} {\ensuremath{\mathcal{R}e}}
\renewcommand{\Im} {\ensuremath{\mathcal{I}m}}

\newcommand{\Vol} {\ensuremath{\textnormal{Vol}}}

\newcommand{\tM} {\ensuremath{\tilde{M}}}
\newcommand{\tX} {\ensuremath{\tilde{X}}}

\newcommand{\tx} {\ensuremath{\tilde{x}}}

\newcommand{\talpha} {\ensuremath{{\tilde\alpha}}}
\newcommand{\tbeta} {\ensuremath{{\tilde\beta}}}

\newcommand{\tomega} {\ensuremath{{\tilde\omega}}}

\newcommand{\nt} {\ensuremath{{\tilde\nabla}}}
\newcommand{\piTN} {\ensuremath{{\vartheta_{TN}}}}

\newcommand{\rhotfirst} {\ensuremath{{\rho_t^{(1)}}}}
\newcommand{\rhotsecond} {\ensuremath{\rho_t^{(2)}}}

\newcommand{\tilderhotfirst} {\ensuremath{\tilde\rho_t^{(1)}}}
\newcommand{\tilderhotsecond} {\ensuremath{\tilde\rho_t^{(2)}}}
\newcommand{\piJN} {\ensuremath{{\vartheta_{J^2N}}}}

\newcommand{\alphap} {\ensuremath{{\alpha^{[\p]}}}}
\newcommand{\alphak} {\ensuremath{{\alpha^{[\k]}}}}
\newcommand{\omegap} {\ensuremath{{\omega^{[\p]}}}}
\newcommand{\omegak} {\ensuremath{{\omega^{[\k]}}}}
\newcommand{\Fp} {\ensuremath{{F^{[\p]}}}}
\newcommand{\Fk} {\ensuremath{{F^{[\k]}}}}
\newcommand{\xip} {\ensuremath{{\xi^{\p}}}}
\newcommand{\xik} {\ensuremath{{\xi^{\k}}}}

\newcommand{\Cinfty} {\ensuremath{\calC^\infty}}

\newcommand{\covt} {\ensuremath{\frac{D}{\partial t}}}

\newcommand{\dev}[2] {\ensuremath{\frac{\partial{#1}}{\partial{#2}}}}
\newcommand{\devd}[2] {\ensuremath{\frac{\partial^2{#1}}{\partial{#2}^2}}}
\newcommand{\scal}[2] {\ensuremath{\langle #1, #2 \rangle}}

\newcommand{\bigscal}[2] {\ensuremath{\big\langle #1, #2 \big\rangle}}

\newcommand{\Bigscal}[2] {\ensuremath{\Big\langle #1, #2 \Big\rangle}}

\newcommand{\bigscalC}[2] {\ensuremath{\big\langle #1, #2 \big\rangle_{\C}}}
\newcommand{\BigscalC}[2] {\ensuremath{\Big\langle #1, #2 \big\rangle_{\C}}}
\newcommand{\normC}[1] {\ensuremath{\|#1\|_{\C}^2}}
\newcommand{\bignormC}[1] {\ensuremath{\big\|#1\big\|_{\C}^2}}

\newcommand{\xiup} {\ensuremath{\xi_1^p}}
\newcommand{\xidp} {\ensuremath{\xi_2^p}}
\newcommand{\etaup} {\ensuremath{\eta_1^p}}
\newcommand{\etadp} {\ensuremath{\eta_2^p}}

\newcommand{\GIT} {\ensuremath{/\!/}}

\newcommand{\cont}{\mathbin{\raisebox{\depth}{\scalebox{1}[-1]{$\lnot$}}}}

\newcommand*{\textrightarrow}[1]{\xrightarrow{\mathmakebox[1.5em]{#1}}}

\theoremstyle{plain}
\newtheorem{theorem}{Theorem}[section]
\newtheorem{lemma}[theorem]{Lemma}
\newtheorem{prop}[theorem]{Proposition}
\newtheorem{cor}[theorem]{Corollary}

\newtheorem{theoremintro}{Theorem}

\newtheorem*{theorem*}{Theorem}
\newtheorem*{lemma*}{Lemma}
\newtheorem*{prop*}{Proposition}
\newtheorem*{cor*}{Corollary}

\theoremstyle{remark}

\newtheorem{remark}[theorem]{Remark}
\newtheorem*{remark*}{Remark}

\theoremstyle{definition}
\newtheorem{example}[theorem]{Example}
\newtheorem{defn}[theorem]{Definition}
\newtheorem{notation}[theorem]{Notation}

\newtheorem*{defn*}{Definition}

\usepackage[english]{babel}